\documentclass[a4paper,11pt]{article}
\usepackage[dvips]{graphicx}
\usepackage{subfigure}
\usepackage{amsmath}
\usepackage[latin1]{inputenc}
\usepackage[english]{babel}
\usepackage{natbib}
\usepackage{lscape}
\usepackage{color}
\usepackage{bm}
\usepackage{multirow}
\usepackage{pdflscape}
\usepackage{lscape}
\usepackage{rotating}
\usepackage{amsfonts}
\usepackage{amssymb}
\usepackage{longtable}
\usepackage{natbib}
\usepackage{mathpazo}
\usepackage[table,xcdraw]{xcolor}
\usepackage[normalem]{ulem}
\usepackage{cancel}

\newtheorem{theorem}{Theorem}

\newtheorem{corollary}{Corollary}

\newtheorem{remark}{Remark}

\newenvironment{proof}[1][Proof]{\textbf{#1.} }{\ \rule{0.5em}{0.5em}}

\makeatletter

\@addtoreset{example}{section}

\@addtoreset{lemma}{section}

\@addtoreset{theorem}{section}

\@addtoreset{table}{section}

\@addtoreset{corollary}{section}

\@addtoreset{equation}{section}

\makeatother
\definecolor{darkgreen}{rgb}{0.0, 0.2, 0.13}
\definecolor{cadmiumgreen}{rgb}{0.0, 0.42, 0.24}
\newcommand{\ren}[1]{{\color{blue}#1}}
\newcommand{\rui}[1]{{\color{cadmiumgreen}#1}}

\begin{document}
\title{Ruin and dividend measures in the renewal dual risk model}
\date{}
\maketitle
\begin{center}
\vspace{-10mm}
 Renata G.~Alcoforado$\, ^1$, Agnieszka I.~Bergel$\, ^2$, Rui~M.R.~Cardoso$\,^3$, Alfredo D.~Eg\'{\i}dio dos Reis$\,^2$, Eugenio V.~Rodriguez-Martinez$\,^4$,\\

\vspace{0.5cm}
{\it
$^1\,$ISEG \& CEMAPRE, Universidade de Lisboa \& Universidade Federal de Pernambuco\\
alcoforado.renata@gmail.com\\
$^2\,$ISEG and CEMAPRE, Universidade de Lisboa\\
agnieszka@iseg.ulisboa.pt, alfredo@iseg.ulisboa.pt\\
$^3\,$CMA and FCT, Universidade Nova de Lisboa\\
rrc@fct.unl.pt\\
$^4\,$CEMAPRE \& Fidelidade Seguros, ISEG, Universidade de Lisboa\\
evrodriguez@gmail.com\\
}
\end{center}
%
\begin{abstract}
In this manuscript we consider the dual risk model with financial application, where the random gains occur under a renewal process. We particularly work the Erlang($n$) case for  common distribution of the inter-arrival times, from there it is easy to understand that our method or procedures can be generalised to other cases under the matrix-exponential family case. We work several and different problems involving future dividends and ruin. We also show that our results are valid even if the usual income condition is not satisfied.

In most known works under the dual model, the main target under study have been the calculation of expected discounted future dividends and optimal strategies, where the dividend calculation have been done on aggregate. We can find works, at first using the classical compound Poisson model, then  some examples of other renewal Erlang models. Knowing that ruin is ultimately achieved, we find important that dividends should be evaluated on an individual basis, where the early dividend contribution for the aggregate are of utmost importance. From our calculations we can really see  how much important is the contribution of the first dividend. \cite{afonso2013dividend} had worked similar problems for the classical compound Poisson dual model.

Besides that we find explicit formulae for both the probability of getting a dividend and the distribution of the amount of a single dividend. We still work the probability distribution of the number of gains to reach a given upper target (like a constant dividend barrier) as well as for the number of gains down to ruin. 

We complete the study working some illustrative numerical examples that show final numbers for the several problems under study. 

\ \\


\end{abstract}

\newpage
\section{Introduction and background}\label{s:intro}
In this manuscript we work the dual risk model with financial application where we specifically work different quantities related with ruin and dividend problems, besides working a ``{non standard}'' approach in the expected amounts of future dividends. Our base model is the renewal dual risk model where the ``standard'' income condition is not necessarily satisfied, as our results are robust to that. As far as the dividend expectation problem is concerned, we consider  an upper barrier level, but this time we don't focus on the expected discounted dividends as an aggregate amount, e.g.~as done in \cite{bergel2017dividends} among many authors. Instead, we take the approach by \cite{afonso2013dividend} done for the classical compound Poisson model and extend for other renewal models.  Like them we compute the expected aggregate discounted amount as a sum of expected discounted individual dividends and consider not only the first moment but also higher moments that are calculated recursively. Then we work other quantities such as the probability of reaching the upper barrier before ruin occurring, interpreted as the probability of reaching a dividend, the single dividend amount and its distribution. Similarly to \cite{egidio2002many} in the case of the classical Poisson insurance risk model, we work in the renewal dual model the distribution of the number of gains down to ruin and the number of gains to reach a given upper target or a dividend level.
In these two problems we work their probability functions. Also, we go back to \textit{standard} dual the model where no upper barriers are set, we mean  a portfolio with dividend free, and that the ruin level is not an absorbing barrier (the process will keep on even if it falls below zero).

%
 
We analytically work the Erlang$(n)$ renewal dual model and get an interesting set of explicit formulae for the problems involved, and it is easily understood that the same methods can be extended for other renewal models under the matrix-exponential family of distributions.   

Let's introduce the dual risk model, basic definitions and notation.
In the (renewal) dual risk model as
described, for instance, by \cite{avanzi2007optimal}, the surplus or equity of a
company at time $t$ is driven by equation, 
\begin{equation}
U(t)=u-ct+S(t)\text{, \ }t\geq 0\text{,}  \label{eq:dual}
\end{equation}
where $u$ is the initial surplus, $c$ is the constant rate of expenses, 
$S(t)=\sum_{i = 0}^{N(t)}X_i$, $X_0\equiv 0$, is the aggregate gains up to time $t$, $N(t)$ is number of gains up to time $t$, $\left\{ N(t),t\geq 0 \right\}$ is a renewal process, so that $\left\{ S(t),t\geq 0 \right\} $ is a compound renewal process. 
The sequence $\{X_i\}_{i=1}^\infty$ is a sequence of independent and identically distributed (shortly i.i.d., with $X\overset{d}{=}X_i$) random variables, each representing the individual gain amount, with 
common distribution function $P(\cdot)$, with $P(0)=0$, density function $p(\cdot)$ and Laplace transform $\hat{p}(\cdot)$. Shortly, we will write d.f.~, p.d.f.~and L.T., respectively. Survival function will be sometimes used and is denoted as $\bar{P}(x)=1-P(x)$. Similar notation may be used for other distribution functions.

$N(t)$ can be written as $N(t) = \max\{k:\, W_1 + W_2 + \cdots + W_k \leq t\}$, where $W_i$ is the waiting time between the ($k-1$)-th and the $k$-th {arrivals}, $k=2,\dots$, and $W_1$ is the first arrival time. We have a sequence of i.i.d. inter-arrival times $\{W_i\}_{i= 1}^\infty$, with $W\overset{d}{=}W_i$, and is independent of the sequence $\{X_i\}$. Common d.f., p.d.f.~and L.T.~  of  the inter-arrival times are respectively denoted as $K(\cdot)$, with $K(0)=0$, $k(\cdot)$ and $\hat{k}(\cdot)$. L.T.'s of some function will be always denoted with a ``$hat$'' in the corresponding function notation.

We assume the existence of the expectations $E(W)$ and $E(X)$, and we set $\mu=E(X)$. Relating this model to the standard insurance model (shortly, Primal Model) we just refer that in  Equation~\eqref{eq:dual} the signs of $c t$ and $S(t)$ are reversed, and we can keep all other assumptions. However, note that in any of those models for the respective application to make sense economically we need to set the so called \textit{net profit condition}. In each application it is reversed, that is, for the dual model the condition comes
\begin{equation}
c\, E(W) < E(X)\, , \label{eq:netprofit}
\end{equation}
whereas in the Primal Model the sign is reversed (``$>$''). Mathematically, many results are robust to the condition, taking into consideration that the ruin probability is one, exactly,  if the income condition is not satisfied in each model. 

In some parts of the text we will need to assume higher moments of $X$, other than $\mu$, in that case we will make it clear in the text. In such situations we  denote by $p_k={E} [X^k]$, $k=1,\dots$, so that $p_1=\mu$.

Let us continue introducing some quantities and notations for further developments, related with the dual model (unless stated clearly otherwise). First, consider the model completely free of barriers and define the Time to Ruin, denoted as $T_u$:
\begin{equation}
	T_{u} := \inf\{t > 0:\; U(t) = 0 |U(0)=u \}, \; u \geq 0\,
	\label{eq:ruintime}
\end{equation}
For a constant $\delta (\geq 0)$, the L.T. of the time to ruin, denoted as $\psi(u,\delta)$, is defined as 
\begin{equation}
\psi(u,\delta) := E(e^{-\delta T_{u}}\mathbb{1}(T_{u} < \infty)). \label{eq:psi_u_delta}
\end{equation}
Constant $\;\delta\;$ can be interpreted as an interest force. The probability of ultimate ruin, defined as $\psi(u) := P(T_{u} < \infty)$, can be obtained as $\psi(u)=\lim_{\delta \downarrow 0} \psi(u,\delta)$. Note that when the income condition is put in force $T_u$ is a defective random variable such that $\mathbb{P}[T_{u} = \infty]>0$. 
 If the income condition \eqref{eq:netprofit} is not fulfilled then  $\psi(u)=1$, and $T_u$ is a proper random variable. In any case, with or without putting in force the income condition, we get $\psi(u) = 1$ if an upper constant barrier $b\geq u$ is set for dividend purposes. We are going to consider this sort of barrier in the following.  

For an arbitrary upper level $b \geq u \geq 0$ define 
\begin{equation*}
\tau_{u} := \inf\{t > 0: U(t) > b \mid U(0) = u\}
\end{equation*}
as the time to reach $b$ for the surplus process, allowing the process to continue even if it crosses the ruin "0" level. When the net profit condition \eqref{eq:netprofit} is applied $\tau_{u}$ is a proper random variable since the probability of crossing $b$ is one. It is a defective random variable otherwise. 

Let us now consider the existence of both the upper dividend barrier and the lower absorbing ruin barrier.  The process, starting from $u$, will reach one of the barriers at some instant in the future with probability one. We define $\chi(u,b)$ as the probability of reaching an upper barrier level $b$ before ruin occurring, for a process with initial surplus $u$, and $\xi(u,b):= 1-\chi(u,b)$ is the probability of ruin before reaching $b$. We have $\chi(u,b) = Pr(\tau_{u} < T_{u})$.

Let $D_{u} := \{U(\tau_{u}) - b\}$ be the dividend amount, from initial surplus $u$, and $\{\tau_{u} <T_{u}\}  $  the event that allows the dividend to occur.  The (defective) distribution function of $\{D_u \leq x  \wedge \tau_{u} <T_{u}\}$ is denoted as $G(u,b;x)$, 
\begin{equation}
	\label{eq:Gubx}
\left. 
	\begin{array}{c}
	G(u,b;x) =  \mathbb{P}{(\tau_{u} < T_{u} \;\;\textrm{and}\;\; U(\tau_{u})} \leq b+x)\mid u,b)  \\
  \quad \quad = \mathbb{P}{(\tau_{u} < T_{u} \;\;\textrm{and}\;\; D_{u} \leq x)\mid u,b)}     
	\end{array}
		\right., \quad (0 < u \leq b)\quad x\geq 0\,,
\end{equation}
with $G(u,b;0)=0$ and respective density function $g(u,b;x) = \frac{d}{dx}G(u,b;x)$. Note that if $u=0$ event $\{\tau_{u} > T_{u}\}  $ is certain then $G(0,b;x)=0$. Also, if $b \leq u \leq b+x$ we have that $D_{u} = u-b$ and $G(u,b;x)=1$. 
 Function $G(u,b;x)$ in \eqref{eq:Gubx} is a defective distribution function when the income condition \eqref{eq:netprofit} is set (proper, otherwise). Clearly,
\begin{equation*}
\lim_{x \rightarrow \infty}G(u,b;x) = \mathbb{P}(\tau_{u} < T_{u}) = \chi(u,b) \,.
\end{equation*}

Also, on aggregate, let $D(u,b,\delta )$ denote the aggregate amount of discounted dividends, at rate $\delta>0$, initial surplus $u$ and dividend barrier $b$, and define $V_k(u;b,\delta):=\mathbb{E}[D(u,b,\delta )^k]$, $k\in \mathbb{N}$, as its $k$-th ordinary moment, with $V(u;b,\delta)=V_1(u;b,\delta)$ for simplicity sake.

 For the model without upper barriers, let us define $q(u,m)$ as the probability of having exactly $m$ gains prior to ruin, given initial surplus $u\,(\geq 0)$. We consider that if no gain arrives ruin occurs at time $t_0=u/c$. $q(u,m)$ is defined for $m=0,1,2,\dots\,$ \ 
 Also, define $r(u,b,m)$, $m=1,2,\dots$, as the probability of having exactly $m$ gains to reach an upper target $b$, like an upper barrier or a dividend barrier, given initial surplus $u \, (\geq 0)$, irrespective of ruin. When the target is reached it is exactly at the instant of a gain arrival, obviously, at least one gain is needed. 
 %
 
%

\

In ruin probability calculation, Lundberg's equations are of utmost importance, in either dual or primal risk model. For our renewal risk model the generalized Lundberg's equation can be written as
\begin{equation}\label{eq_Lund}
\hat{k}(\delta-cs)\hat{p}(s)=1 ,\; s\in \mathbb{C}\,. 
\end{equation}
If we let $\delta \downarrow 0$ we get the standard or fundamental Lundberg's equation. If $W_i$, $ i=1,2,\dots$, is Phase-Type($n$) distributed then both the generalized and fundamental  equations have $n$ roots with positive real parts, if income condition \eqref{eq:netprofit} is fulfilled. However, if \eqref{eq:netprofit} is not fulfilled the fundamental will have only $n-1$ of such roots (this corresponds to the primal model ruin problem situation described in \citep{li+garrido04a}). We will get back to this discussion later, when appropriately relevant, in Section~\ref{s:single_div}. The $n$ roots in each equation (generalised or fundamental) may not be all  distinct. For instance, 
\begin{itemize}
	\item  
If $W\frown \text{Erlang}(n,\lambda)$ then all roots are distinct;
	\item If $W\frown \text{generalised Erlang}(n,\lambda)$ we can have double roots; and 
	\item We can have higher root multiplicity for other kinds of distributions in the Phase-Type family.
\end{itemize}
 We can find a discussion on this in \cite{bergel2014sparre}, \cite{bergel2015further}, \cite{bergel2016ruin} and \cite{bergel2017dividends}. 
 
We remark (somehow again) that in spite of addressing sometimes the primal model, in this paper we work the dual risk model only, where the income condition {\eqref{eq:netprofit} usually applies}, giving economic sense to the application. These two models address different problems, however, there are many contact points: Many mathematical developments are robust to the condition, and some problems have many similarities, like the importance of the Lundberg's equations, and the dividends and its distribution in the dual having similarities with the severity of ruin in the primal.

We will be summarising first some known problems in the literature and then our developments for the different ruin and dividend problems. Some findings are more general than others, some are general for renewal risk models, some will need the distribution of the inter-arrival time, $W$, to be disclosed, some even need the distribution of single gain, $X$, to be assumed. We remark that when we particularize the distribution  of $W$, we use the {$\text{Erlang}(n,\lambda)$} example to show how we can find a solution. 
In that case we have that $K(t) =1- \sum_{i=0}^{n-1} e^{-\lambda t}\left( \lambda t \right)^i / i!$,  and $k(t)={\lambda^n t^{n-1}e^{-\lambda t}} / {(n-1)!}$, $t\geq 0$.
Similar approaches can be used and extended for more general matrix exponential distributions.
Taking into account the works by \cite{bergel2016ruin} and \cite{bergel2017dividends} for instance, the  techniques used can be extended straightforwardly for the generalized Erlang or even for a general Phase--Type$(n)$ distribution. 

We particularly recall that regarding the {$\text{Erlang}(n,\lambda)$} law, $n\in \mathbb{N}$,  we will make use some well known properties: for $n=2,3,\dots$, where respective d.f.~and p.d.f., $K(.)$ and $k(.)$ are sometimes conveniently denoted as $K_n(.)$ and $k_n(.)$ to enhance {the} identifying parameter,
\begin{eqnarray}
k'_n(t) &=&\lambda\left(k_{n-1}(t)-k_n(t)\right)\,; \nonumber \\
k^{(i)}_n(0) & =& 0\,, \text{ for } i=0,1,\dots,n-2; \label{eq:k-derivatives} \\
k^{(n-1)}_n(0) &=&\lambda^n. \nonumber
\end{eqnarray} 

For instance, for the expected discounted ruin probability in \eqref{eq:psi_u_delta}, we retrieve from \cite{rodriguez2015some} the general renewal equation, with $t_0=u/c$, 
\begin{equation}
	\psi (u,\delta )=\left( 1-K_{W}\left( t_{0}\right) \right) e^{-\delta
		t_{0}}
+\int_{0}^{t_{0}}k_{W}\left( t\right) e^{-\delta t}\int_{0}^{\infty
	}p(x)\psi (u-ct+x,\delta )dxdt\, \text{.}  \label{eq:psi_u_delta2}
\end{equation}
From there, particularizing for $Erlang(n,\lambda)$ dual model, with $W\frown \text{Erlang}(n)$, we get the integro-differential equation (briefly IDE) in the next theorem, where \textsc{$\mathcal{D}$} is the differential operator, $\mathcal{D}^i$  responds for the $i$-th derivative,  with the identity operator $\mathcal{D}^0=\mathcal{I}$. If we find necessary we write in the index the operand variable, e.g.~in the following it would be $\mathcal{D}_u=\frac{d}{du}$ and $\mathcal{D}_u^k=\frac{d^k}{du^k}$. 
%
\begin{theorem}
	\label{Theor_intdif1}
	In the Erlang($n,\lambda$) dual risk model the Laplace
	transform of the time of ruin satisfies the integro--differential equation 
	\begin{equation*}
	\left( \left( 1+\frac{\delta }{\lambda }\right) \mathcal{I}+\left( \frac{c}{%
		\lambda }\right) \mathcal{D}\right) ^{n}\psi (u,\delta )=\int_{0}^{\infty
	}p(x)\psi (u+x,\delta )dx\text{,}  \label{eq6}
	\end{equation*}%
	with boundary conditions 
	\begin{equation*}
	\psi (0,\delta )=1,\quad \displaystyle\left. \frac{d^{i}}{du^{i}}\psi
	(u,\delta )\right\vert _{u=0}=(-1)^{i}\left( \frac{\delta }{c}\right)
	^{i},\;i=1,\ldots ,n-1\text{.}  
	\label{eq9}
	\end{equation*}
\end{theorem} 
\begin{theorem}
	\label{Theor_psiu}{\normalsize In the Erlang($n,\lambda$) dual risk model, the Laplace transform of the time of ruin can
		be written as a combination of exponential functions} 
	\begin{equation*}
	\psi (u,\delta )=\sum_{k=1}^{n}\left( \prod_{i=1,i\neq k}^{n}\frac{\rho _{i}-%
		\frac{\delta }{c}}{\rho _{i}-\rho _{k}}\right) e^{-\rho _{k}u}\text{,}
	\label{eq7}
	\end{equation*}%
	where $\;\rho _{1},\ldots ,\rho _{n}\;$ are the only roots of the
	generalized Lundberg's equation which have positive real parts.
\end{theorem}

We show that similar approaches can be used for other ruin and dividend problems as we'll study in the next sections.
The manuscript evolves as follows. In Section~\ref{s:expect_div} we work expectations of discounted future dividends, in Section~\ref{s:single_div} we develop the distribution of a single dividend amount as well as the probability of getting a single dividend. Section~\ref{s:no_ruin} and \ref{s:no_div} we work with gain counts down to ruin and up to dividend, respectively. Finally, in the last section we work some illustrative numercial examples.

\section{Expected Discounted Dividends}
\label{s:expect_div}
In this section we calculate expectations for discounted future dividends, on aggregate or considering each single dividend.  On aggregate we could retrieve formulae like those worked by \cite{avanzi2007optimal}, \cite{cheung+drekic08},  \cite{ng09}, \cite{ng10}, or we can use the method which separates each individual dividend, properly discounted,  introduced by \cite{afonso2013dividend} for the {$\text{Erlang}(1,\lambda)$} case (classical compound Poisson). We get closed formulae.

First, we need to consider two barriers: one reflecting and one absorbing. The reflecting barrier is an upper dividend barrier, an arbitrary level non-negative value $b$. If $b<u$ the surplus immediately is set at $b$, so this an uninteresting situation. We mostly consider the case where $b \geq u \geq 0$. The absorbing barrier is the ruin level zero. We refer to the upper graph of Figure~1 from \cite{afonso2013dividend}.

For a general compound renewal dual model, we can generalise the formulae for the total expected dividends from \cite{afonso2013dividend}. Although it was developed for the classical compound Poisson model it is easy to see that its application is more general. The {expected} total discounted dividends $V(u;b,\delta )$ is given in the following theorem,
\begin{theorem} For the renewal dual risk model
	\begin{eqnarray*}
V(u;b,\delta) & = & \mathbb{E}[D(u,b,\delta )]=\mathbb{E}\left[ \sum\limits_{i=1}^{%
	\infty }e^{-\delta \left( \sum_{j=1}^{i}T_{(j)}\right) }D_{(i)}\right] \text{
	, }0\leq u\leq b\text{ ,}\\
&=&\mathbb{E}\left( e^{-\delta \tau_{u}}D_{u}\right) +\mathbb{E}\left(
e^{-\delta \tau_{u}}\right) \frac{\mathbb{E}\left( e^{-\delta
		\tau_{b}}D_{b}\right) }{1-\mathbb{E}\left( e^{-\delta \tau_{b}}\right) }\text{,}
\end{eqnarray*}
where $T_{(1)}=\tau_u$, $D_{(1)}=D_u$, $T_{(i)}$ and $D_{(i)}$, {for} $i=2,3,\dots$, are single dividend and dividend waiting time $i$, or replicas of $\tau_b$ and $D_b$, respectively.
\end{theorem}
	\begin{proof}
	Immediate, using the same arguments as those of Formula~(4.2) from \cite{afonso2013dividend}. \hfill
	\end{proof}

%
%
For higher moments, $V_n(u;b,\delta)$ recursion given by (4.7-8) in  \cite{afonso2013dividend} also applies similarly, we reproduce,
\begin{theorem}
	For the renewal dual risk model
	\begin{equation}
\label{eq:vnub_2}
V_{n}(u;b,\delta )=\sum_{k=0}^{n}\binom{n}{k}\mathbb{E}\left[ e^{-n\delta
	\tau_{u}}{D_{u}}^{k}\right] V_{n-k}(b;b,\delta )\,\text{,}
\end{equation}%
with $V_{0}(b;b,\delta )=1$ and
\begin{equation}
\label{eq:vnbb}
V_{n}(b;b,\delta )=\frac{\sum_{k=1}^{n}\binom{n}{k}\mathbb{E}\left[
	e^{-n\delta \tau_{b}}{D_{b}}^{k}\right] V_{n-k}(b;b,\delta )}{1-\mathbb{E}\left[
	e^{-n\delta \tau_{b}}\right] }\,.
\end{equation}
\end{theorem}

Now, we need to develop appropriate formulae for the discounted single dividend expectations. Denote by $\phi_k(u)=\mathbb{E}\left[ e^{-\delta \tau_{u}}{D_{u}}^{k}\right]$, with $k=0,1,2,\dots$ Note that for $u>b$, $\phi_k(u)=(u-b)^k$. We have the following theorem:
	%
	\begin{theorem}\label{thm_intdif1}
		In the Erlang($n,\lambda$) dual risk model, $\phi_k(u)$ satisfies  the integro--differential equation
		\begin{equation}\label{eq_ide_phik}
		\left( \left( 1+\frac{\delta }{\lambda }\right) \mathcal{I}+\left( \frac{c}{%
			\lambda }\right) \mathcal{D}\right) ^{n}\phi_k(u) =\omega{_{\phi}}(u) 
		\end{equation}%
		with
		\begin{align} \label{eq_Wu}
		\begin{split}
		\omega
		{_{\phi}}(u) &= 
		\int_{0}^{b-u}\phi_k(u+y)p(y)dy + \int_{b-u}^{\infty
		} (u+y-b)^kp(y)dy\\
		&
		{=\int_{u}^{b}\phi_k(y)p(y-u)dy + \int_{b}^{\infty
		} (y-b)^kp(y-u)dy}
		\text{,} 
		\end{split}
		\end{align}
		with boundary conditions 
		\begin{equation}
		\label{eq_Wubc}
		\quad \displaystyle\left. \phi_k^{(i)}
		(u)\right\vert _{u=0}=0\,,\;\;i=0,1,\ldots ,n-1\text{.}  
		\end{equation}
	\end{theorem}
	\begin{proof} 
	First we note that for the general  renewal dual risk model, we have, conditioning on the time and amount of the first gain, that
\begin{eqnarray*}
	\phi_k(u) &= & \int_0^{t_0} k \left( t\right) e^{-\delta t} \times \\
	 && \text{ \ \ } \left\{\int_{0}^{b-(u-ct)} \phi_k(u-ct+y) p(y)dy +   \int^{\infty}_{b-(u-ct)} (u-ct+y-b)^k p(y)dy \right\} dt\text{,} 
\end{eqnarray*}
where $t_0$ is such that $u-ct_0=0$. Changing the integration variable, $s=u-ct$, we get
 $$
 \phi_k(u)=\frac{1}{c} \int_0^{u} k\left( {(u-s)}/{c}\right) e^{-\delta\left(   \frac{u-s}{c}\right)} \omega{_{\phi}}(s) ds\,,
 $$
 with $\omega{_{\phi}}(s)$ given by \eqref{eq_Wu}.
 
Now, we particularly assume that $W\frown \text{Erlang}(n,\lambda)$ distribution. For the remaining of the proof we use an inductive argument and so we conveniently write $k(.)=k_{{i}}(.)$, since we need to consider and denote differently shape parameter $i$ of Erlang$(i,\lambda)$ distributions, of their  p.d.f.~or d.f.'s $k_i(\cdot)$ and $K_i(\cdot)$, respectively, as well as corresponding $\phi_k(u)=\phi_{k,i}(u)$, $i=1,2,\dots$ \ In what follows we use properties~\eqref{eq:k-derivatives}.

Now, we differentiate with respect to $u$ and get, for $n=2,3,\dots$, 
\begin{eqnarray*}
\phi_{k,n}'(u)&=& \frac{1}{c} \int_0^{u}  \left[ \frac{1}{c} \, k_n' \left( \frac{u-s}{c}\right)
-\frac{\delta}{c} \,k_n \left( \frac{u-s}{c}\right) \right] e^{-\delta\left(   \frac{u-s}{c}\right)} \omega_{\phi}(s) ds\\
&=&\frac{\lambda}{c}\, \phi_{k,n-1}(u)-\frac{\delta +\lambda}{c}\, \phi_{k,n}(u)
\end{eqnarray*}
or equivalently,
$$
\left( \left( 1+\frac{\delta }{\lambda }\right) \mathcal{I}+\left( \frac{c}{%
	\lambda }\right) \mathcal{D}\right) \phi_{k,n}(u)= \phi_{k,n-1}(u) \,.
$$
Assuming that, for $i=2,\dots n-1$
$$\left( \left( 1+\frac{\delta }{\lambda }\right) \mathcal{I}+\left( \frac{c}{%
	\lambda }\right) \mathcal{D}\right)^i \phi_{k,n}(u)= \phi_{k,n-i}(u)\,,$$
we get
\begin{eqnarray*}
\left( \left( 1+\frac{\delta }{\lambda }\right) \mathcal{I}+\left( \frac{c}{%
	\lambda }\right) \mathcal{D}\right)^{i+1} \phi_{k,n}(u)
&=& \left( \left( 1+\frac{\delta }{\lambda }\right) \mathcal{I}+\left( \frac{c}{%
	\lambda }\right) \mathcal{D}\right)\phi_{k,n-i}(u)\\
&& \hspace{-2cm} =\left( 1+\frac{\delta }{\lambda }\right)\phi_{k,n-i}(u)+\frac{c}{\lambda}\, \left[ 
\frac{\lambda}{c}\phi_{k,n-i-1}(u)-\frac{\delta +\lambda}{c}\,\phi_{k,n-i}(u)
\right]\\
&& \hspace{-2cm}={\phi_{k,n-(i+1)}(u)\,,}
\end{eqnarray*}
particularly for $i=n-1$,
$$
\left( \left( 1+\frac{\delta }{\lambda }\right) \mathcal{I}+\left( \frac{c}{%
	\lambda }\right) \mathcal{D}\right)^{n-1} \phi_{k,n}(u)= \phi_{k,1}(u)
$$
where $\phi_{k,1}(u)$ has derivative
$$
\phi_{k,1}'(u)=-\frac{\delta +\lambda}{c}\, \phi_{k,1}(u)+\frac{\lambda}{c}\, \omega{_{\phi}}(u)\,.
$$
Hence,
$$
\left( \left( 1+\frac{\delta }{\lambda }\right) \mathcal{I}+\left( \frac{c}{%
	\lambda }\right) \mathcal{D}\right)^{n} \phi_{k,n}(u)=\left( \left( 1+\frac{\delta }{\lambda }\right) \mathcal{I}+\left( \frac{c}{%
	\lambda }\right) \mathcal{D}\right) \phi_{k,1}(u)=\omega{_{\phi}}(u) \,.
$$
Taking successive derivatives of $\phi_{k,n}(u)$, we find that
\begin{equation*}\label{eq_phidev_i}
{\phi_{k,n}^{(i)}(u)}=\sum _{j=0}^i (-1)^{i+j} \binom{i}{j}\left(\frac{\lambda +\delta}{c}\right)^{i-j}\left(\frac{\lambda}{c}\right)^j \phi_{k,n-j}(u)
\end{equation*}
for $i=0,1,\ldots, n-1$, we obtain the boundary conditions. We note that we have $\phi_{k,n}(0)=0$ as ruin is certain, \textit{almost surely}.
\hfill
\end{proof}

%
	{We can obtain an expression for $\phi_k(u)$ via Laplace transforms by using the method presented in \cite[end of Section~3]{afonso2013dividend}. First we define $\tilde{\phi}_k(z):=\phi_k(b-z)=\phi_k(u)$, replacing $u$ by $z=b-u$. This change of variable is equivalent to switch from the dual model to the primal one. Afterwards, we extend the domain of the latter defined function from $[0,b]$ to $[0,\infty [$ and denote the L.T.~of the resulting function by $\widehat{\phi}_k(s)$, whose expression is stated in the next theorem. Then, if possible, we invert the L.T.~and, at last, we revert the change of variable.}
	\begin{theorem}
		%
		In the Erlang($n,\lambda$) dual risk model, the Laplace Transform 	$\widehat{\phi}_k(s)$ 
{is given by
%
%
%
\begin{equation}\label{eq_phik_hat}
\widehat{\phi}_k(s) = \frac{ \sum_{i=1}^{n} \binom{n}{i} (\lambda+\delta)^{n-i} (-c)^i \sum_{j=0}^{i-1} s^{i-1-j} {\tilde{\phi}}_k^{(j)}(0) + \lambda^n p_k \frac{1-\widehat{p}_k(s)}{s} }
{(\lambda + \delta - c s)^n - \lambda^n \widehat{p}(s)}\,,
\end{equation}
}
where $\hat{p}_k(s)$ is the L.T.~of the $k$-th equilibrium density relative to $p(.)$. 
	\end{theorem}
%
\begin{proof} 
With the change of variable, $z=b-u$,  the IDE {\eqref{eq_ide_phik}} becomes
	\begin{equation*}
\left( {\left( 1 + \frac{\delta }{\lambda }\right)} \mathcal{I}-\left( \frac{c}{%
	\lambda }\right) \mathcal{D}_z\right) ^{n} {\tilde{\phi}}_k(z) =\omega\rui{_{\tilde{\phi}}}(z)
\end{equation*}%
with
\begin{equation*} \label{eq_Wu_LT}
\omega{_{\tilde{\phi}}(z)} = 
\int_{0}^{z} {\tilde{\phi}}_k(z-y)p(y)dy + \int_{z}^{\infty
} (y-z)^kp(y)dy\text{,} 
\end{equation*}
{and} boundary conditions 
\begin{equation*}
\label{eq_Wubc_LT}
\quad \displaystyle\left. {\tilde{\phi}}_k^{(i)}
(z)\right\vert _{z=b}=0\,,\;\;i=0,1,,\ldots ,n-1\text{.}  
\end{equation*}

Rewriting the new IDE, we get
$$
\sum_{i=0}^{n} \binom{n}{i} \left(1+\frac{\delta}{\lambda}\right)^{n-i}\left(-\frac{c}{\lambda}\right)^{i} \tilde{\phi}_k^{(i)}(z)=\omega {_{\tilde{\phi}}(z)}\,.
$$
Taking Laplace transforms to both sides, we get the result, noting that
$$ \int_{0}^{\infty}e^{-sz} {\tilde{\phi}}_k^{(i)}(z)dz=s^i\widehat{\phi}_k(s) - \sum_{j=0}^{i-1} s^{i-1-j}  {\tilde{\phi}}_k^{(j)}(0)\,, $$
and that
$$\int_{0}^{\infty} e^{-sz}  \int_{z}^{\infty
} (y-z)^kp(y)dy dz= p_k \frac{1-\widehat{p}_k(s)}{s} \,,
$$
{see \cite[end of Section~3]{afonso2013dividend} and references therein, for both the formula and the details on equilibrium distributions.}
\hfill
\end{proof} 	

\begin{remark} In Formula \eqref{eq_phik_hat} we note that
	\begin{enumerate}
		\item For $k=0$ we have that $p_0=1$ and  $\widehat{p}_0(s)=\widehat{p}(s)$.
		\item For $n=1$ we got the result already stated in \cite[Formula (3.6)]{afonso2013dividend}.
		\item To invert the L.T.~we need to find the zeros of the denominator, they correspond to the roots of the generalized Lundberg's equation in \eqref{eq_Lund}. In the process, we need to use boundary conditions~\eqref{eq_Wubc} to find {$\tilde{\phi}_k^{(j)}(0)=(-1)^j {\phi}_k^{(j)}(b^-)$.}
	\end{enumerate}
\end{remark}

\section{On the amount and the probability of a single dividend}
\label{s:single_div}
\subsection{Introduction}
In this section we present two methods to calculate the distribution of a dividend amount,  $G(u,b;x)$ as defined in \eqref{eq:Gubx}. From there we can calculate straightforwardly the probability of reaching an upper barrier before ruin.
As before, we particularize, exemplifying the case when the gain inter-arrival times follow an Erlang($n,\lambda$) distribution. 

 For a general renewal risk model, if no gain arrives before $t_{0}$ necessarily {the} event $\{\tau_{u} < T_{u}\}$ is impossible,  and so 
$G(u,b;x)=0$.
If the first gain arrives before $t_{0}$, either it does or does not cross $b$, then we derive a defective renewal equation, for $0\leq u\leq b$ and $x\geq0$, 
\begin{eqnarray*}
	G(u,b;x) &=& \int_{0}^{t_{0}}k(t)\left( {\int_{0}^{b-(u-ct)}G(u-ct+y,b;x)p(y)dy}\right. 
+ \left. {\int_{b-(u-ct)}^{b+x-(u-ct)}p(y)dy} \right)dt \, ,
\end{eqnarray*}
where the first inner integral represents the probability of having a first gain at a fixed time $t$, with the corresponding amount not crossing the dividend level however   happening in the future. The second integral represents the probability of a dividend at the instant of the first gain $t$.
 
 With the change of variable $s=u-ct$ the defective renewal equation above becomes
\begin{equation}\label{eq:Gubx2}
G(u,b;x) = \frac{1}{c}\int_{0}^{u}k\left(\frac{u-s}{c}\right)\omega_{G}(s,b;x)ds,\quad 0\leq u\leq b\,, 
%
\end{equation}
{where}
\begin{eqnarray}
	\omega_{G}(s,b;x) & = & \int_{0}^{b-s}G(s+y,b;x)p(y)dy + \int_{b-s}^{b+x-s}p(y)dy \nonumber \\
	&=& \int_{s}^{b}G(y,b;x)p(y-s)dy + \int_{b}^{b+x}p(y-s)dy\,. \label{eq:W_Gsbx} 
\end{eqnarray}

Considering the Erlang$(n,\lambda)$ dual risk model an IDE for $G(u,b;x)$ follows in the next theorem.  
\begin{theorem}\label{th_G(u,b;x)_ide}
	In the Erlang$(n,\lambda)$ dual risk model, $G(u,b;x)$ satisfies the integro--differential equation 
\begin{equation}\label{eq:G(u,b;x)_ide}
\left(\mathcal{I}+\left(\frac{c}{\lambda}\right)\mathcal{D}_u\right)^{n}G(u,b;x) = \omega_{G}(u,b;x),\quad 0<u<b\,,
\end{equation}
where $\omega_{G}(u,b;x)$ is given by \eqref{eq:W_Gsbx},  {with} boundary conditions
\begin{equation}\label{chap5:eq3}
\mathcal{D}^{i}_uG(u,b;x)\left.\right|_{u=0}=0, \quad  i=0,1,\dots,n-1\,.
\end{equation}
%
%
\end{theorem}
\begin{proof}
	Take successive derivatives with respect to $u$ or by induction similarly to that of Theorem~\ref{thm_intdif1}.
	\hfill
\end{proof}

To calculate a solution for $G(u,b;x)$ from equations above is not straightforward. We can do it in two ways: Either using the annihilator method used by \cite[Section~5.2]{rodriguez2015some} or the Laplace transform method like is used by \cite{afonso2013dividend}.  We work both in the following subsections.
\subsection{Annihilator method}
The annihilator is a polynomial operator when applied to a given density $p(x-u)$ annihilates function $\omega_{G}(u,b;x)$ in \eqref{eq:G(u,b;x)_ide}, see \cite[Section~5.2]{rodriguez2015some}. Their method found a solution considering a $p(x)$ within the Phase-Type distribution family, briefly $PH(m)$, for a similar function but for a different problem (in the case expected discounted dividends in the dual model). We can bring their method into our problem, and the same distribution family for which we have $n+m$ roots of Lundberg's fundamental equation. However, there are some specific aspects in our problem that are worth to be developed. Applying the method we reach the solution given in the next theorem.  
\begin{theorem}\label{th_G(u,b;x)_cond}
	In the Erlang$(n,\lambda)$ dual risk model, 	$G(u,b;x)$ satisfies, for $0<u\leq b$,
	\begin{equation}\label{eq:Gubx_sum}
	G(u,b;x)=\sum_{i=0}^{n+m-1}a_{i}e^{-r_{i}u},
	\end{equation}
	where, for each $i=0,\dots, n+ m-1$, $r_{i}$   is  a root of Lundberg's {fundamental} equation and the coefficient $a_{i}$, function of $b$ and $x$,  is found using boundary conditions \eqref{chap5:eq3} together with the additional constraint
	\begin{equation}\label{chap5:eq3_add}
	\sum_{i=0}^{n+m-1}\left[a_{i}\int_{b-u}^{\infty}e^{-r_{i}y}p(y)dy\right]e^{-r_{i}u}-P(b+x-u) + P(b-u) = 0\, , \; \forall\, 0<u\leq b.
	\end{equation}
\end{theorem}
\begin{proof}
	 Similar to the method used by~\cite[Section 5.2]{rodriguez2015some}, consider a polynomial operator, an annihilator of degree $m$ for $p(y-u)$ in $\mathcal{D}=\frac{d}{du}$, $A(\mathcal{D})=\sum_{j=0}^{m} b_j\frac{d^j}{du^j}$,  where $b_j$, $j=0,1,\dots,m$, are constants.   $A(\mathcal{D})$ is such that
\begin{equation} \label{eq_annil_py-u}
A(\mathcal{D})p(y-u)=0.
\end{equation}
Then from \eqref{eq:G(u,b;x)_ide} we apply the operator
\begin{equation}\label{eq_annilG(u,b;x)_ide}
A(\mathcal{D})\left(\mathcal{I}+\left(\frac{c}{\lambda}\right)\mathcal{D}\right)^{n}G(u,b;x)=A(\mathcal{D})\omega_{G}(u,b;x)\,,
\end{equation}
in order to find an homogeneous integro-differential equation of some degree $m+n$, whose solutions can be written as combinations of exponential functions
\begin{equation}\label{eq:Gubx_sum2}
G(u,b;x)=\sum_{i=0}^{n+m-1}a_{i}e^{-r_{i}u},
\end{equation}
for some constants $r_{i}$ and some coefficients $a_{i}$, $i=0,1,\dots,n+m-1$, functions of $b$ and $x$, all independent from $u$.

Taking into account~\eqref{eq_annil_py-u}, the application of the annihilator  to $\omega_G(.)$ leads to
\begin{equation*}
A(\mathcal{D})\omega_{G}(u,b;x)= A(\mathcal{D})\int_{u}^{b}G(y,b;x)p(y-u)dy\,.
\end{equation*}
 We have that the $j$-th derivative with respect to $u$ of the above integral comes, for $j=1,2,\dots$,
 \begin{equation*}
 \mathcal{D}^j\int_{u}^{b}G(y,b;x)p(y-u)dy=\sum_{k=0}^{j-1}(-1)^{k+1}\mathcal{D}^{j-1-k}G(u,b;x)\,p^{(k)}(0) + \int_{u}^{b}G(y,b;x)\, \mathcal{D}^j p(y-u)dy\,,
 \end{equation*}
 %
 so that 
 \begin{eqnarray*}
 A(\mathcal{D})\int_{u}^{b}G(y,b;x)p(y-u)dy &= & \sum_{j=0}^{m} b_j \left[\sum_{k=0}^{j-1}(-1)^{k+1}\mathcal{D}^{j-1-k} G(u,b;x) \,p^{(k)}(0)\right] \\
 &=& \sum_{k=0}^{m-1}  \left[ \sum_{j=k+1}^{m} b_j(-1)^{j-k} \,p^{(j-k-1)}(0) \right]
 \mathcal{D}^{k} G(u,b;x)
  \,.
 \end{eqnarray*}
   since  $\int_{u}^{b}G(y,b;x)\, A(\mathcal{D})\, p(y-s)dy=0$. If we write the inner summation above as $\tilde{b}_k$, we can write \eqref{eq_annilG(u,b;x)_ide} as
  \begin{equation*}
  A(\mathcal{D})\left(\mathcal{I}+\left(\frac{c}{\lambda}\right)\mathcal{D}\right)^{n}G(u,b;x)
  -  \sum_{k=0}^{m-1}  \tilde{b}_k \mathcal{D}^{k} G(u,b;x)= 0\,,
  \end{equation*} 
   which is an homogeneous integro-differential equation of degree $m+n$ with solutions of the form~\eqref{eq:Gubx_sum2}, as required.
   
To find {the} constants and {the} coefficients  we replace {\eqref{eq:Gubx_sum2}} in \eqref{eq:G(u,b;x)_ide} and arrive to the conclusion that the constants $r_i$'s must be the roots of  Lundberg's fundamental equation. Then, the coefficients $a_i$, $i=0,1, \dots,n+m-1$, can be found using the boundary conditions \eqref{chap5:eq3} together with additional constraint \eqref{chap5:eq3_add}. This can be shown as follows.

We have, with $a={\lambda}/{c}$, 
\begin{eqnarray*}
\left(\mathcal{I}+\left(\frac{c}{\lambda}\right)\mathcal{D}\right)^{n}G(u,b;x)&=& 
\sum_{k=0}^{n} \binom{n}{k}a^{-k}\mathcal{D}^k_u
\left(\sum_{i=0}^{n+m-1}a_{i}e^{-r_{i}u}\right) \\
& = &
\sum_{i=0}^{n+m-1}a_{i}e^{-r_{i}u} \sum_{k=0}^{n} \binom{n}{k} {\left(-a^{-1} r_i\right)^{k}}
\\
&=& \sum_{i=0}^{n+m-1}a_{i}e^{-r_{i}u} \left(1-r_i/a\right)^n= 
\sum_{i=0}^{n+m-1}a_{i}e^{-r_{i}u}\hat{p}(r_i)
\,,
\end{eqnarray*}
using Lundberg's fundamental equation.

Replacing {\eqref{eq:Gubx_sum2}} in the right-hand side of \eqref{eq:G(u,b;x)_ide} we get 
\begin{equation*}
\omega_{G}(u,b;x)=\sum_{i=0}^{n+m-1} \, a_{i}\, e^{-r_{i}u}\int_{0}^{b-u}e^{-r_{i}y}p(y)dy+P(b+x-u) - P(b-u) \, .
\end{equation*}
Equating the two sides we get
\begin{equation*}
\sum_{i=0}^{n+m-1} \, a_{i}\, e^{-r_{i}u}\left(\hat{p}(r_i)-\int_{0}^{b-u}e^{-r_{i}y}p(y)dy\right) = P(b+x-u) - P(b-u) \,,
\end{equation*}
and  condition \eqref{chap5:eq3_add} follows. 
\hfill
\end{proof}
\begin{remark}
	The extra $m$ boundary conditions can be found from constraint \eqref{chap5:eq3_add} as follows.
$$A(\mathcal{D})=\sum_{j=0}^{m} b_j\frac{d}{du}$$
 is the annihilator of $p(y-u)$. Then, $A(\mathcal{D})\omega_{G}(u,b;x)=0$, in particular 
 $$
 A(\mathcal{D}) \left[P(b+x-u) + P(b-u)\right]=0\,.
 $$
 From condition \eqref{chap5:eq3_add}
$$
\sum_{i=0}^{n+m-1}\left[a_{i}\int_{b-u}^{\infty}e^{-r_{i}y}p(y)dy\right]e^{-r_{i}u}=P(b+x-u) - P(b-u)\, , \; \forall\, 0<u\leq b\,,
$$
we have, for $ 0<u< b$
$$
A(\mathcal{D})\sum_{i=0}^{n+m-1}\left[a_{i}\int_{b-u}^{\infty}e^{-r_{i}y}p(y)dy\right]e^{-r_{i}u}= A(\mathcal{D})\left[P(b+x-u) - P(b-u)\right]=0\,.
$$
Hence, we can write, for $k=0,1,\dots,m-1$, and $0<u < b$.
$$
\mathcal{D}^k_u
\left(
\sum_{i=0}^{n+m-1}\left[a_{i}\int_{b-u}^{\infty}e^{-r_{i}y}p(y)dy\right]e^{-r_{i}u}
\right)= \mathcal{D}^k_u \left[P(b+x-u) - P(b-u)\right]\, .
$$

Since conditions are valid for all $0<u<b$,
we can set $u=b^-$ and calculate
$$
\mathcal{D}^k_u
\left.
\sum_{i=0}^{n+m-1}
\left[ a_{i}\int_{b-u}^{\infty}e^{-r_{i}y}p(y)dy
\right]
e^{-r_{i}u}
\right|_{u=b^-}= 
\mathcal{D}^k_u \left.\left[P(b+x-u) - P(b-u)\right]\right|_{u=b^-}\, 
$$
simplifying significantly the calculations for finding the $a_i$'s coefficients.
\end{remark}
\begin{remark}
	From \eqref{eq:W_Gsbx} and Theorem~\ref{th_G(u,b;x)_ide} we write that 
	\begin{equation*}
	\omega_G(b,b;x)=P(x)
	\end{equation*}
	and that 
	\begin{eqnarray*}
		\lim_{u\uparrow b} \left(\mathcal{I}+{a^{-1}}\, \mathcal{D}_u\right)^{n}G(u,b;x) &=&P(x)\\
		{\lim_{u\uparrow b} }\sum_{k=0}^{n} \binom{n}{k}a^{-k}\mathcal{D}^k_u \, G(u,b;x)&=& P(x)\,.
	\end{eqnarray*}
	Using \eqref{chap5:eq3_add}, above equation is equivalent to 
	\begin{eqnarray*}
		\sum_{i=0}^{n+m-1} a_{i}\hat{p}(r_i)\,e^{-r_{i}b} &=&P(x) \\
		\sum_{i=0}^{n+m-1}a_{i}\left(1-r_i/a \right)^n\, e^{-r_{i}b} &=&P(x) \,,
	\end{eqnarray*}
	since $\hat{p}(r_i)= \hat{k}(-c\,r_i)^{-1}$ from fundamental Lundberg's Equation~\eqref{eq_Lund}, with $\delta=0^+$.  
\end{remark}
To get a final formula for $G(u,b;x)$ we must  find the $n+m$ roots  of Lundberg's equation, either with positive, or negative real parts, or even with existing null roots (there is at least one). See remark that follows.
\begin{remark}
	Referring to Theorems~\ref{th_G(u,b;x)_ide} and \ref{th_G(u,b;x)_cond}, in any of the developments therein the usual income {condition~\eqref{eq:netprofit}} is imposed. Indeed, results remain valid if the income condition is reversed. In any case looking at the fundamental Lundberg's equation, we know that there is always a null solution. Note that according to \cite[Theorem~2 and Remark~1]{li+garrido04a} and \cite[Section~3]{rodriguez2015some} we conclude that, clearly, 
	\begin{enumerate}
		\item In the case of a negative loading condition, $c\, E(W) < E(X)\,$ we have exactly $n$ roots with positive real parts, $m-1$ with negative real parts and a null root;
		\item In the case of a positive loading condition we have exactly $n-1$ roots with positive real parts, $m$ with negative real parts and a null root;
		\item If $cE(W) = E(X)\,$ the sequence will be $n-1$, $m-1$ and a double null root.  
	\end{enumerate} 
Because we'll need to take derivatives, we need to take care that in either formulae \eqref{eq:Gubx_sum} or \eqref{chap5:eq3_add}, we have one constant $a_i$ that is independent of the corresponding exponential factor $e^{-r_i\,u}$ (it is two in {the} third case above).
\end{remark}

\subsection{Laplace transform method}
Here, we can follow the approach presented by \cite{afonso2013dividend} and get the {Laplace} transform that follows in the next theorem. First, define 
\begin{equation}\label{chap5:eq5}
\tilde{G}(z,b;x):=G(b-z,b;x)=G(u,b;x)\,,
\end{equation}
with $z=b-u \Leftrightarrow u=b-z$,  and
\begin{equation}
\label{eq:rho_func}
\rho(z,b;x):=\left\{\begin{array}{cc}
\tilde{G}(z,b;x), & 0 < z < b\, . \\
0, & z \geq b, 
\end{array}
\right.
\end{equation}
extending the domain of $\tilde{G}(z,b;x)$, as a function of $z$, to $(0, \infty)$.
\begin{theorem}
	The Laplace transform of $\rho(z,b;x)$is given by 
	\begin{equation}\label{chap5:eq8}
	\widehat{\rho}(s,b;x)=\frac{\sum_{j=1}^{n}\binom{n}{j} \left(-{c}\right)^j{\lambda}^{n-j} \sum_{i=0}^{j-1}\rho^{(j-1-i)}(0,b;x)\, s^{i}+
		{\lambda}^{n} \left[\mathcal{T}_{s}P(x)-\mathcal{T}_{s}P(0)\right]}{\left({\lambda}-{c}s\right)^{n}-{\lambda}^n\widehat{p}(s)},
	\end{equation}
	where $\mathcal{T}_{s}f(x) = \int_{0}^{\infty}e^{-st}f(t+x)dt$ is an integral operator over an integrable function $f$, known in the actuarial literature as the Dickson-Hipp operator, see \cite{dickson2001Onthetime}. 
	%
\end{theorem}
\begin{proof}
Following the same line of reasoning in Theorem 2.4, we have, changing the variable, that the integro-differential equation \eqref{eq:G(u,b;x)_ide} becomes,
\begin{equation*}
\left(\mathcal{I}-\left(\frac{c}{\lambda}\right)\frac{d}{dz}\right)^{n}\tilde{G}(z,b;x) =\int_{0}^{z}\tilde{G}(z-y,b;x)p(y)dy + P(z+x)-P(z) ,\quad 0<z<b,
\end{equation*}
with boundary conditions
\begin{equation*}
\frac{d^{i}}{dz^{i}}\tilde{G}(z,b;x)\left.\right|_{z=b}=0, \quad  i=0,1,\dots, n-1.
\end{equation*}
Applying the L.T.~ to each side of the IDE, it comes the final result because
$$
\int_{0}^{\infty}e^{-sz}P(z+x)dz=:\mathcal{T}_{s}P(x)$$
and
$$ 
\int_{0}^{\infty}e^{-sz}P(z)dz=\mathcal{T}_{s}P(0)=\frac{\hat{p}(s)}{s}\,.
$$
%
%
%
%
%
%
%
\hfill
\end{proof}

Note that for $n=1$ Formula~\eqref{chap5:eq8} corresponds to Result~(6.5) in \cite{afonso2013dividend}.
Similarly to what is remarked concerning to Formula~\eqref{eq_phik_hat}, if we want to obtain $\rho(z,b;x)$, and therefore $G(u,b;x)$, we must specify a distribution $P(x)$ for the  gain amounts. This allows to factor the de\-no\-mi\-na\-tor, using all the roots of the {fundamental} Lundberg's equation, separate the re\-sul\-ting expression into partial fractions and then invert the L.T.~\eqref{chap5:eq8}.

\section{On the number of gains down to ruin}
\label{s:no_ruin}
 In this section we work the probability $q(u,m)$ as the probability of having exactly $m$ gains prior to ruin, given initial surplus $u$. We consider that if no gain arrives, ruin occurs at time $t_0=u/c$. If a gain arrives before time $t_0$, at time $t \in (0,t_0)$,  surplus will be immediately after $U(t)=u-ct+X_1$. If $m$ gains arrive, $m=2,3,\dots$,  the first has to occur before $t_0$ necessarily, say $t$, as well as the following $m-1$, after the renewal of the process at $t$. Then we have for $u>0$
\begin{eqnarray}
	q(u,0) &=& 1-K (u/c) \nonumber\\
	q(u,m) &=& \int_{0}^{u/c}k (t)\int_{0}^{\infty}q(u-ct+x,m-1)p(x)dxdt\,, \text{ \ } m=1,2,\dots\,. \label{eq:qum}
\end{eqnarray}
For $u=0$, the process is ruined immediately after start (\textit{almost surely}), then we can write that $q(0,0)=1$ and $q(0,m)=0$, $m=1,2,\dots$

For the $Erlang(n,\lambda)$ dual risk model we can set the following integro-differential equation, in the theorem below. In the theorem we conveniently use the notation, {for $i=0,1,2,\dots n-1$},
\begin{eqnarray}
q_{n-i}(u,m) &=& \int_{0}^{u/c}k_{n-i}(t)\int_{0}^{\infty}q(u-ct+x,m-1)p(x)dxdt  \nonumber \\[2.0mm]
& =& \frac{1}{c}\int_{0}^{\infty}\int_{0}^{u}k_{n-i}\left(\frac{u-s}{c}\right)q(s+x,m-1)p(x)dsdx
\,,\text{ \ }  m=1,2,\dots \,, \label{eq:qum_ni}
\end{eqnarray} 
so that {$q_{n   }(u,m)=q(u,m)$}.
 \begin{theorem}\label{theo_gain_no}
	In the Erlang$(n,\lambda)$ dual risk model, $q(u,m)$ given by \eqref{eq:qum}  satisfies the integro--differential equation 
\begin{equation}\label{eq:gain_no_ide}
\left(\mathcal{I}+\left(\frac{c}{\lambda}\right)\mathcal{D}_u\right)^{n}q(u,m)=  \omega_{q}(u,m-1)\, \quad n, m \in\mathbb{N} ;  \,, 
\end{equation}
where $\omega_{q}(u,x)$ is given by 
\begin{equation}\label{eq:qum_W}
\omega_{q}(u,x)=\int_{u}^{\infty}q(y,x)p(y-u)dy \,,
\end{equation} 
%
with boundary conditions
\begin{equation}\label{eq:qum_difn}
\mathcal{D}^{i}_u q(u,m))\left.\right|_{u=0}=0, \quad  i=0,1,\dots,n-1\,; m\in \mathbb{N}\,.
\end{equation}
 \end{theorem}
\begin{proof}
Using a standard approach, setting $s=u-ct$, for $n=2,3,\dots$ Equation~\eqref{eq:qum} can be written as 
\begin{eqnarray*}
q(u,m) &=& \int_{0}^{\infty}  \int_{0}^{u} k_n\left(\frac{u-s}{c}\right)q(s+x,m-1)ds p(x)dx\,, \text{ \ } m=1,2,\dots \,, 
\end{eqnarray*} 
differentiating with respect to $u$, using properties~\eqref{eq:k-derivatives} and noting \eqref{eq:qum_ni}, we can write 
 \begin{equation*}
 \left(\mathcal{I}+\left(\frac{c}{\lambda}\right)\mathcal{D}_u\right)q(u,m)=  {q}_{n-1}(u,m-1)\,. 
 \end{equation*}
 Recursively, we can get 
  \begin{equation*}
 \left(\mathcal{I}+\left(\frac{c}{\lambda}\right)\mathcal{D}_u\right)^{n-1}q(u,m)=  {q}_{1}(u,m-1) \,,
 \end{equation*}
 noting that $k_1(\cdot)$ is an Exponential($\lambda$) density. Then, 
  \begin{equation*}
 \left(\mathcal{I}+\left(\frac{c}{\lambda}\right)\mathcal{D}_u\right)^{n}q(u,m)=  \left(\mathcal{I}+\left(\frac{c}{\lambda}\right)\mathcal{D}_u\right){q}_{1}(u,m-1)=\omega_q(u,m-1) \,.
 \end{equation*}
 For $n=1$, it follows immediately from the first step, or last above, that  $\left(\mathcal{I}+\left(\frac{c}{\lambda}\right)\mathcal{D}_u\right)q(u,m)=  \omega_q(u,m-1)$.
 
 For the boundary conditions, note that 
 \begin{equation*}
  \left(\mathcal{I}+\left(\frac{c}{\lambda}\right)\mathcal{D}_u\right)^{i}q(u,m)= \sum_{k=0}^{i}
  \binom{i}{k}\left( \frac{c}{\lambda}\right)^k \mathcal{D}_u^k \, {q}_{n-i}(u,m-1) \,,
 \end{equation*}
 then
 \begin{equation*}
\lim_{u\downarrow 0} \left(\mathcal{I}+\left(\frac{c}{\lambda}\right)\mathcal{D}_u\right)^{i}q(u,m)=  \sum_{k=0}^{i} \binom{i}{k}
 \left( \frac{c}{\lambda}\right)^k \left. \mathcal{D}_u^k  \, {q}_{n-i}(u,m-1) \right|_{u=0}=0\,, \quad i=0,1,\dots, n-1.
\end{equation*} 
 Recursively, we get \eqref{eq:qum_difn}.
 \hfill
\end{proof}

We can find  an easy Laplace transform formula for \eqref{eq:qum}, see next theorem.
\begin{theorem}
	In the Erlang$(n,\lambda)$ dual risk model the  Laplace transform for ${q}(u,m)$, denoted as $\hat{q}(s,m)$, is given by 
\begin{eqnarray}
\hat{q}(s,m) &= & \frac{\widehat{\omega}_{q}(s,m-1)}{\left(1+\frac{c}{\lambda } s\right)^n }\,,  \quad m=1,2,\dots \label{eq:qumLT}
\\[1.1mm]
 \hat{q}(s,0) & = & \widehat{\Bar{K}}_{n,a}(s) = \frac{1}{s}\left[ 1- \hat{k}_{n,a}(s)\right] \label{eq:qu0LT}
\,.
\end{eqnarray}
where $\widehat{\omega}_{q}(s,x)$ is the L.T.~ of \eqref{eq:qum_W},  $\Bar{K}_{n,a}(s)=1-{K}_{n,a}$ and ${K}_{n,a}$ are respectively the survival function and corresponding density, where subscript $\{n,a\}$ refers to  updated scale parameter $a=\lambda/c$ of the $Erlang(n,a)$ distribution.
\end{theorem}
\begin{proof}
	Develop \eqref{eq:gain_no_ide}, apply the L.T.~ operator and its properties, to get
	\begin{equation*}
\sum_{k=0}^{n}\binom{n}{k} a^{-k} \widehat{\mathcal{D}^{k}_u} q(s,m)= 
\sum_{k=0}^{n}\binom{n}{k} a^{-k} \left[ s^k \, \hat{q}(s,m) - \sum_{j=1}^{k-1} s^j \, q^{(k-1-j)} \,(0,m)\right]\,,
	\end{equation*}
	where $q^{(i)} \,(0,m)=\left. \mathcal{D}^{i}_u q(u,m)\right|_{u=0}$.  Getting use of the boundary conditions \eqref{eq:qum_difn}, we have that  $q^{(i)} \,(0,m)=0$, $i=0,\dots, k-1$, and obtain
	\begin{equation*}
	\hat{q}(s,m)\left(1+a^{-1} s\right)^n=\widehat{\omega}_{q}(s,m-1)\,.
	\end{equation*} 
	
	Formula~\eqref{eq:qu0LT} is immediate since 	$q(u,0) = 1-K_{n,\lambda}(u/c)=1-K_{n,a}(u)$, then using L.T.~ properties.
	\hfill
\end{proof}

In sequence of Theorem~\ref{theo_gain_no}, from Equation \eqref{eq:gain_no_ide} we can develop an homogeneous differential equation in order to provide a solution to that result. This is done in the  theorem that follows.
{
\begin{theorem}\label{th:gain_no_hdifeq}
	In the Erlang$(n,\lambda)$ dual risk model, $q(u,m)$ given by \eqref{eq:qum}  satisfies the homogeneous differential equation, on $q(u,m)$ and $q(u,m-1)$, $m=1,2,\dots$,
	\begin{equation}\label{eq:gain_no_hdifeq}
	\mathcal{A}(\mathcal{D})
	\left(\mathcal{I}+\left(\frac{c}{\lambda}\right)\mathcal{D}_u\right)^{n}q(u,m)
	-
	\sum^{m^\ast-1}_{k=0} \tilde{b}_k \mathcal{D}^k_u q(u,m-1) = 0 \, \quad n, m \in\mathbb{N} ;  \,, 
	\end{equation}
	where $\tilde{b}_k$ is given by 
	\begin{equation*}
	\tilde{b}_k = \sum^{m^\ast}_{j=k+1} b_j(-1)^{j-k} p^{(j-1-k)}(0) \,,
	\end{equation*}
	$\mathcal{A}(\mathcal{D})$ is the annihilator operator of $p(y-u)$, $m^\ast$ the respective polynomial degree,
	\begin{equation*}
	\mathcal{A} (\mathcal{D}) =\sum^{m^\ast}_{j=0} b_j \mathcal{D}^j\,,
	\end{equation*}
	$\mathcal{D}$ is the differential operator %
	and $b_j$, $j=0,1,\dots,k$, some constants.
\end{theorem}
\begin{proof}
From the right-hand side of \eqref{eq:gain_no_ide}, we write
\begin{eqnarray*}
	\mathcal{A} (\mathcal{D}) \omega_{q}(u,m-1) &=&  	\mathcal{A} (\mathcal{D}) \int_{u}^{\infty}q(x,m-1)p(x-u)dx \\ 
	& = &\sum^{m^\ast-1}_{k=0}  \left[ \sum^{m^\ast}_{j=k+1} b_j(-1)^{j-k} p^{(j-1-k)}(0) \right] \mathcal{D}^k_u \, {q}(u,m-1)  \\
	&= & \sum^{m^\ast-1}_{k=0} \tilde{b}_k \mathcal{D}^k_u \, q(u,m-1) \,.
\end{eqnarray*}
\hfill
\end{proof}

Equation \eqref{eq:gain_no_hdifeq} is an homogeneous differential equation on $q(u,m)$ and $q(u,m-1)$ that can be solved recursively, since we know that $q(u,0)= 1-K_n(u/c)$. 

For computing $\omega_{q}(u,0)$ and $\widehat{\omega}_{q} (s,0)$,  needed in~\eqref{eq:gain_no_ide}, \eqref{eq:qum}  and  \eqref{eq:qumLT} as starting values, we can write the following theorem:
	\begin{theorem}\label{th_Wq+hatWq}
		In the Erlang$(n,\lambda)$ dual risk model, $\omega_{q}(u,0)$ and $\widehat{\omega}_{q} (s,0)$ are given by, 
\begin{eqnarray}
	\omega_{q}(u,0)
	&=&    \sum_{i=0}^{n-1} (-1)^{i}\, \frac{a^i}{i!} \mathcal{D}^i_s \left.
	\left(  \hat{p}(s) e^{-s\,u}  
	\right)\right|_{s=a} \,, \nonumber \\[1.5mm]
	&= & 
	 \sum_{i=0}^{n-1}   \frac{a^i}{i!} 
	\left(	\sum_{j=0}^{n-i-1} (-1)^j  \, \frac{a^j}{j !}\hat{p}^{(j)}(a)\right)\,u^{i}  \, e^{-a\,u}   \,,   \label{eq_Wqn}
\end{eqnarray}
and
\begin{equation} \label{eq_LTWqn}
\widehat{\omega}_{q} (s,0)
=   \sum_{i=0}^{n-1}   \frac{a^i}{i!} 
\left(	\sum_{j=0}^{n-i-1} (-1)^j  \,\frac{a^j}{j!} \hat{p}^{(j)}(a)\right)\, \frac{1}{(a+s)^{i+1}} \,,
\end{equation}
where  $a=\lambda/c$ and $ \left. \mathcal{D}^i_s \hat{p}(s)\right|_{s=a}= \hat{p}^{(i)}(a)$.
	\end{theorem}
\begin{proof}
	From \eqref{eq:qum_W}
\begin{eqnarray*}
	\omega_{q}(u,0) &= & \int_{0}^{\infty}{q}(x+u,0)\,p(x)dx \\
	&=& \int_{0}^{\infty}\left(  \sum_{i=0}^{n-1} e^{-a(u+x)} \frac{a^i}{i!}(x+u)^i \right) p(x)dx\\
	& & \\
	&=&   \sum_{i=0}^{n-1} e^{-a\,u} \frac{a^i}{i!}\sum_{k=0}^{i} \binom{i}{k} u^{i-k}   \int_{0}^{\infty}e^{-a\,x} x^k p(x)dx \,.
\end{eqnarray*}
setting $(x+u)^i= \sum_{k=0}^{i} \binom{i}{k} x^k u^{i-k}$. Using derivatives of the L.T.~of $p(x)$ evaluated at $a$, $ \left. \mathcal{D}^k_s \hat{p}(s)\right|_{s=a}= \hat{p}^{(k)}(a)$ and $\hat{p}^{(0)}(a)=\hat{p}(a)$, we write
\begin{eqnarray*}
	\omega_{q}(u,0)
	&=&   \sum_{i=0}^{n-1} e^{-a\,u} \frac{a^i}{i!}\sum_{k=0}^{i} \binom{i}{k} u^{i-k}   (-1)^k \hat{p}^{(k)}(a) \,.
\end{eqnarray*}
Using Leibnitz's rule for derivatives of product we get first expression for $\omega_{q}(u,0)$ in \eqref{eq_Wqn}. 

Let us locally denote that function as $\omega_{q, n} (u,0)$ to underline its  correspondence to Er\-lang$(n,\cdot)$  case.  We can calculate it recursively, for $n=1,2,\dots $, as follows,
\begin{eqnarray*}
 \omega_{q, 1} (u,0) &= & e^{-a\, u} \hat{p}(a)\,; \\
 \omega_{q, n} (u,0) & = & \omega_{q, n-1} (u,0) + (-1)^{n-1} \,  \frac{a^{n-1}}{(n-1)!} \, \left. \mathcal{D}^{n-1}_a 
 \left(  \hat{p}(s) e^{-s\,u}  \right)\right|_{s=a}\,.
\end{eqnarray*}
%
Following that, developing, we can also write it as in \eqref{eq_Wqn}. From there it is  immediate that 
the corresponding L.T.~corresponds to  \eqref{eq_LTWqn}
$\hfill$
\end{proof}

We can find a solution for $q(u,1)$  solving Integro-differential Equation~\eqref{eq:gain_no_ide}, so that we can write  the  theorem that follows.
\begin{theorem}\label{th_gain_no}
	In the Erlang$(n,\lambda)$ dual risk model, $q(u,1)$ is given by
	\begin{eqnarray} 
q(u,1) &= & \sum_{j=0}^{n-1} B_j \,  u^{n+j}\, e^{-a\,u} \label{eq_qu1} \\
 B_j &= & \frac{a^{n+j}}{(n+j)!} 
 	\left(	\sum_{i=0}^{n-1-j} (-1)^i  \, \frac{a^i}{i!}\hat{p}^{(i)}(a)\right)\,, \quad j=0,1,\dots, n-1, \label{eq_Bj_qu1} \,
	\end{eqnarray}
	and
	$$
\omega_{q, n} (u,1) = 	\sum_{j=1}^{n-1} \, B_j  \, (-1)^{n+j}   \mathcal{D}^{n+j}_s \left.\left( \hat{p}^{(k)}  (s) \,  e^{-s\,u}\right) \right|_{s=a} \,.
$$
\end{theorem}
\begin{proof}
	From \eqref{eq:gain_no_ide} we   write 
	\begin{equation}\label{eq_difn_qu1}
		\left(a\, \mathcal{I}+ \mathcal{D}_u\right)^{n}q(u,1)=  a^n\, \omega_{q}(u,0)  \,, 
		\end{equation}
		with $\omega_{q}(u,0)$ given by \eqref{eq_Wqn} and boundary conditions 	$\mathcal{D}^{i}_u q(u,1))\left.\right|_{u=0}=0$,  $i=0,1,\dots,n-1$.
		
		Let $q_{n,h}(u,1)$ and $q_{n,p}(u,1)$ be the solution of the homogeneous equation, 	
		\begin{equation} \label{eq:gain_no_idehom}
			\left(a\, \mathcal{I}+ \mathcal{D}_u\right)^{n}q(u,1)=0\,,
		\end{equation} 
		and a particular solution of the differential equation above, respectively. Thus, 	$q_{n}(u,1)=q_{n,h}(u,1)+q_{n,p}(u,1)$.

		First, let's deal with solution $q_{n,h}(u,1)$. The characteristic polynomial $(a+r)^n$ has one root $r=-a$ with multiplicity $n$. Hence, 
		\begin{equation*}
	q_{n,h}(u,1) = \sum_{k=0}^{n-1}A_k \, u^k\, e^{-a\,u} \,,
		\end{equation*}
		where $A_k$, $k=0,\dots, n-1$, are some constants, to be found using boundary conditions \eqref{eq:qum_difn}. Rewrite $q_{n,h}(u,1) = \left(A_0+ \sum_{k=1}^{n-1}A_k \, u^	k\right)\, e^{-a\,u}$, then $A_0=0$ using the boundary conditions. We can proceed recursively to find that all $A_k=0$, $k=1,\dots,n-1$:
	\begin{eqnarray*}
		q'_{n,h}(u,1) &= &A_1(1-a\,u)\,e^{-a\,u}+ \sum_{k=2}^{n-1}A_k \,\left( u^k\, e^{-a\,u} \right)' \\
		q"_{n,h}(u,1) &= &A_2(2-4a\,u+ a\, u^2)\,e^{-a\,u}+ \sum_{k=3}^{n-1}A_k \,\left( u^	k\, e^{-a\,u} \right)^"\\%
		& & \cdots \\
		q^{(n-1)}_{n,h}(u,1) &= &A_{n-1} \left((n-1 )! + \sum_{j=1}^{n-1} \frac{(n-1)!}{j!} 
		\binom{n-1}{j}(-a)^j\,u^j\, e^{-a\,u} \right)\,.
	\end{eqnarray*}
      Therefore, $q_{n,h}(u,1)\equiv 0$.

        As far as the particular solution is concerned, $q_{n,p}(u,1)$ must be linearly independent from the functions $e^{a\,u}, u\,e^{a\,u}, \dots, u^{n-1}\,e^{a\,u}$, since they generate  the homogeneous solutions. Also, it depends on the shape of the non-homogeneous component of the differential equation,  
        function $\omega_{q, n} (u,0)$  given by \eqref{eq_Wqn}. Therefore, we propose a solution of the form:
        \begin{equation*}
        	q_{n,p}(u,1) =  \sum_{j=0}^{n-1} B_j u^{n+j} \, e^{-a\,u}\, 
        \end{equation*}
		for some constants $B_j$, $j=0,1,\dots, n-1$. We now replace $q_{n,p}(u,1)$ in \eqref{eq_difn_qu1}, developing the lefthand side, 
		\begin{equation*} 
		\sum_{	l=0}^{n}\binom{n}{l}a^{n-l}q^{(l)}_{n,p}(u,1) =
			a^n\, \sum_{k=0}^{n-1}   \frac{a^k}{k!} 	\left(	\sum_{j=0}^{n-k-1} (-1)^j  \, \frac{a^j}{j !}\hat{p}^{(j)}(a)\right)\,u^{k}  \, e^{-a\,u} \,,
		\end{equation*}
		where the $l$-th derivative is given by
	\begin{equation*}
		q^{(l)}_{n,p}(u,1) = \sum_{j=0}^{n-1} B_j \sum_{i=0}^{l}\frac{(n+j)!}{(n+j-l+i)!} 
		\binom{l}{i}(-a)^i\,u^{n+j-l+i } \, e^{-a\,u} \,,
	\end{equation*}
so that  we get for the lefthand side, also changing sum order, setting $\binom{l}{i}=\binom{l}{l-i}$ and changing the summation variable $r=l-i$,
	\begin{equation}\label{eq_difn_qu1a}
\sum_{	l=0}^{n}\binom{n}{l}a^{n-l}q^{(l)}_{n,p}(u,1)  = 
      \sum_{j=0}^{n-1} B_j \sum_{r=0}^{n} a^{n-r}  \frac{(n+j)!}{(n+j-r)!} \, u^{n+j-r} 
\left[\sum_{i=0}^{n-r} \binom{n}{r+i} \binom{r+i}{r}  (-1)^i \right] \, e^{-a\,u} \,.
\end{equation}
		The sum inside the square brackets is equal to zero for $r\in {0,1,\dots, n-1}$ (see Remark~\ref{rmk_null_sum} that follows) so that 
			\begin{equation*}
			\sum_{	l=0}^{n}\binom{n}{l}a^{n-l}q^{(l)}_{n,p}(u,1)  = 
			\sum_{j=0}^{n-1} B_j   \frac{(n+j)!}{(j!} \, u^j \, e^{-a\,u} \,.
		\end{equation*}
		
		Now, Equation~\eqref{eq_difn_qu1a} becomes,
			\begin{equation*} 
		\sum_{j=0}^{n-1} B_j   \frac{(n+j)!}{j!} \, u^j \, e^{-a\,u} = a^n
		\sum_{k=0}^{n-1}   \frac{a^k}{k!} 	\left(	\sum_{j=0}^{n-k-1} (-1)^j  \, \frac{a^j}{j !}\hat{p}^{(j)}(a)\right)\,u^{k}  \, e^{-a\,u} \,,
		\end{equation*}
		and equating the coefficients we get 
		\begin{equation*} 
	 B_j  =  \frac{a^{n+j}}{(n+j)!} 
	\left(	\sum_{i=0}^{n-1-j} (-1)^i  \, \frac{a^i}{i!}\hat{p}^{(i)}(a)\right)\, \quad  
		j=0,1,\dots, n-1 \,,
		\end{equation*}
		and we get \eqref{eq_qu1}, since $q_{n,h}(u,1)\equiv 0$.
		
		To find $\omega_{q, n} (u,1)$, we start from \eqref{eq:qum_W} that
		\begin{eqnarray*}
		\omega_{q, n} (u,1) & =& \int_{0}^{\infty} \sum_{j=0}^{n-1} \, B_j  \, (x+u)^{n+j} \, e^{-a(x+u)} p(x) \, dx \\
	& =& \sum_{j=0}^{n-1} \, B_j \left[ \sum_{k=0}^{n+j} \binom{{n+j}}{k}\, (-1)^k \hat{p}^{(k)}  (a) \,u^{u+j-k}\,  e^{-au}\right] \\
	 &=& \sum_{j=0}^{n-1} \, B_j  \, (-1)^{n+j}  \, \mathcal{D}^{n+j}_s \left. \left(\hat{p}  (s) \,  e^{-s\,u}\right)\right|_{s=a} \,,
		\end{eqnarray*}
	calculating the $k$-th derivative of the L.T.~and then transforming the  expression inside the square brackets into the $(n+j)$-th derivative of the product $\left( \hat{p}  (s) \right)\,  \left(e^{-s\,u}\right)$.
		$\hfill$
\end{proof}
\begin{remark}\label{rmk_null_sum}
\begin{eqnarray*}
		\sum_{i=0}^{n-r} \binom{n}{r+i} \binom{r+i}{r}  (-1)^i  & = &
		\sum_{i=0}^{n-r} \frac{n!}{(n-r-i)! \, r! \, i!} (-1)^i  \\
		&=& \frac{n(n-1)\dots (n-r+1)}{r!} \, \sum_{i=0}^{n-r}  \binom{n-r}{i}\, 1^{n-r-i}\,(-1)^i \\
		& = & \frac{n(n-1)\dots (n-r+1)}{r!} \, (1-1)^{n-r} =0\, 
\end{eqnarray*}
for all $ n=1,2,\dots$ and $r=0,1,\dots, n-1\,$.
\end{remark}
As a last remark, we can add another expression for $\omega_{q, n} (u,1)$:
\begin{remark}
	We can use induction, for $n=1,2,\dots$,  to show that 
\begin{equation*}
\omega_{q, n} (u,1) = e^{-au}\,\sum_{j=0}^{2n-1}\, \,u^{j}\,  \sum_{i=\max(j,n)}^{2n-1}  \, B_{i-n}  \binom{{i}}{j}\, (-1)^{i-j} \hat{p}^{(i-j)}  (a)  \, .
\end{equation*}
\end{remark}

To calculate the pair $\{ q_n (u,m), \, \omega_{q, n} (u,m)\}$ for higher integer $m$ we can proceed like in the diagram:
$$ \{ {q_n} (u,0)\rightarrow \omega_{q, 0} (u,0)\} \rightarrow
\{ {q_n} (u,1)\rightarrow \omega_{q, n} (u,1)\} \rightarrow
\{ {q_n} (u,2)\rightarrow \omega_{q, n} (u,2)\} \rightarrow \dots \,,
$$
and so on...
\section{On the number of gains to reach a given upper target}
\label{s:no_div}
We work here the probability function $r(u,b,m)$, $m=1,2,\dots$, as the probability of having exactly $m$ gains to reach an upper target $b$, like an upper barrier or a dividend barrier, given initial surplus $u$, irrespective of ruin. When the target is reached it is exactly at the instant of a gain arrival, obviously, at least one gain is needed. For $u\geq 0$, we have
\begin{eqnarray}
	r(u,b,1) &=& \int_{0}^{\infty}k_n(t)[1-P(b-u+ct)]dt \label{rq:rub1}\\
	r(u,b,m) &=& \int_{0}^{\infty}k_n(t)\int_{0}^{b-u+ct}r(u-ct+x,b,m-1)p(x)dxdt\,, \text{ \ } m=2,3,\dots \label{rq:rubm}
\end{eqnarray}
For simplification we set $v=b-u$ so that $r(u,b,m)= r(0,v,m)$, with $b\geq u\geq 0$ and $m\in \mathbb{N}$. Also write $a=\lambda/c$.
 \begin{theorem}\label{theo:gainup_no_ide}
	In the Erlang$(n,\lambda)$ dual risk model, $r(0,v,m)=r(u,b,m)$ given by formula \eqref{rq:rub1}  satisfies the integro--differential equations 
	\begin{eqnarray}\label{eq:theo:gainup_no_ide}
	\left(\mathcal{I}-\left(\frac{c}{\lambda}\right)\mathcal{D}_v\right)^{n} r(0,v,1)=  a\left[1-P(v)\right]\, \quad n, m \in\mathbb{N}  \, . \nonumber
	\end{eqnarray}
	%
	%
	Boundary conditions can be found from, for $i=0,1,\dots,n-1$,
	\begin{eqnarray}\label{eq:r0v1_difn}
\lim_{v\downarrow 0}	\left(\mathcal{I}-\left(\frac{c}{\lambda}\right)\mathcal{D}_v\right)^{i} r(0,v,1)&=& 1-\sum_{k=0}^{n-i-1}\frac{{(-a)}^k}{k!}\mathcal{D}_s^{\,k}
\left.	\hat{p}(s)\right|_{s=a} ds\,. 
	\end{eqnarray}
\end{theorem}
\begin{proof}
Setting $b-u=v$ and $s=v+ct$, $r(0,v,1)$ and its derivative with respect to $v$, $\mathcal{D}_v \, r(0,v,1)$  are
\begin{eqnarray*}
r(0,v,1) & = & c^{-1}\int_{v}^{\infty}k_n\left(\frac{s-v}{c}\right)\bar{P}(s) ds \\
\mathcal{D}_v r(0,v,1) &= & a \left(r(0,v,1)-r_{n-1,n}(0,v,1)\right)\,,
\end{eqnarray*}
where 
\begin{equation*}
r_{n-i,n}(0,v,1)=c^{-1}\int_{v}^{\infty}k_i\left(\frac{s-v}{c}\right)\bar{P}(s) ds\,, \quad \text{for \ } i=0,1,\dots, n-1\,,
\end{equation*}
with $r_{n-n,n}(0,v,1)=r(0,v,1)$.

Recursively we get, for $i=1,2,\dots$,
\begin{eqnarray*}
	\left(\mathcal{I}-\left(\frac{c}{\lambda}\right)\mathcal{D}_v\right)^{i} r(0,v,1) & =&  r_{n-i,n}(0,v,1)\,, \quad \text{and}\,,\\
\left(\mathcal{I}-\left(\frac{c}{\lambda}\right)\mathcal{D}_v\right)^{n} r(0,v,1) & = &	\left(\mathcal{I}-\left(\frac{c}{\lambda}\right)\mathcal{D}_v\right)
r_{1,n}(0,v,1)
=	a\left[1-P(v)\right]\,.
\end{eqnarray*}
	For the boundary conditions, note that for $i=0,1,\dots,n-1$ we have
	\begin{eqnarray*}
\lim_{v\downarrow 0}	\left(\mathcal{I}-\left(\frac{c}{\lambda}\right)\mathcal{D}_v\right)^{i} r(0,v,1)&=& 	r_{n-i,n}(0,0,1)= c^{-1}\int_{0}^{\infty}k_i\left({s}/{c}\right)\bar{P}(s) ds
	\\
	&=& \int_{0}^{\infty}K_i\left({s}/{c}\right){p}(s) ds\\
	&=& 1-\sum_{k=0}^{n-i-1}\frac{{a}^k}{k!}\int_{0}^{\infty}e^{-a\,s} \,s^k p(s) ds\\
	&=& 1-\sum_{k=0}^{n-i-1}\frac{{(-a)}^k}{k!}\mathcal{D}_s^{\,k}
	\left.	\hat{p}(s)\right|_{s=a} ds\,, 
	\end{eqnarray*}
integrating by parts and noting that $\hat{p}(s)=\mathbb{E}[e^{-s\,X}]$ is the Laplace transform of density $p(s)$.
 
\hfill
\end{proof}
\begin{corollary} 
Boundary conditions got from \eqref{eq:r0v1_difn} give, particularly,
		\begin{eqnarray}
r(0,0,1)&=& 1-\sum_{k=0}^{n-1}\frac{{(-a)}^k}{k!}\mathcal{D}_s^{\,k}
\left.	\hat{p}(s)\right|_{s=a} \, ;   \label{eq:rub1_dif0n}\\
r'(0,0,1)&=& \frac{{(-a)}^n}{(n-1)!}\mathcal{D}^{n-1}_s \left.	\hat{p}(s)\right|_{s=a}\,;  \label{eq:rub1_difn} 
%
\end{eqnarray}
For $i=2,\dots,n-1$ and $n=2,3,\dots\,$\,,
\begin{eqnarray} \label{eq:ri001_difn}
r^{(i)}(0,0,1) &=& (-a)^{i}\left[ 1-\sum_{k=0}^{n-1-i} \frac{(-a)^{k}}{k!}\hat{p}^{(k)}(a)-\sum_{k=0}^{i-1}\binom{i}{k}(-a)^{-k} \,  r^{(k)} (0,0,1)\right] \,,
\end{eqnarray}
particularly,
\begin{equation*} \label{eq:rn-1001_difn}
r^{(n-1)}(0,0,1) = (-a)^{n-1}\left[ 1-\hat{p}(a)-\sum_{k=0}^{n-2}\binom{n-1}{k}(-a)^{-k} \, r^{(k)} (0,0,1)\right] \,,
\end{equation*}
where $\,r^{(k)} (0,0,1) = \mathcal{D}_v^{\,k}\left. r (0,v,1)\right|_{v=0}\,$ and $\, \hat{p}^{(k)}(a)= \mathcal{D}_s^{\,k}
\left.	\hat{p}(s)\right|_{s=a} $.
\end{corollary}
\begin{proof}
	Note that 
		\begin{equation*}
		r_{n-i,n}(0,0,1)=  1-\sum_{k=0}^{n-i-1}\frac{{(-a)}^k}{k!}\mathcal{D}_s^{\,k}
		\left.	\hat{p}(s)\right|_{s=a} ds\,, 
	\end{equation*}
	and that $\lim_{v\downarrow 0}	\left(\mathcal{I}-\left(\frac{c}{\lambda}\right)\mathcal{D}_v\right)^{i} r(0,v,1)= \sum_{k=0}^{i}\binom{i}{k}(-a)^{-k} \, r^{(k)} (0,0,1)\,$, then result follows. \hfill
\end{proof}

\begin{theorem}
		In the Erlang$(n,\lambda)$ dual risk model, for $m=2,3, \dots\,$, $r(0,0,m)=r(b,b,m)$ given by formula \eqref{rq:rubm}  can be computed recursively as
	\begin{equation*} \label{eq:r00m}
r(0,0,m) =  \sum_{k=0}^{n-1}\frac{{(-a)}^k}{k!}\mathcal{D}_s^{\,k}
	\left. 	s\,\hat{p}(s) \hat{r}(0,s,m-1)\right|_{s=a}\, , 
	\end{equation*}
	where $r(0,0,1)$ is given by \eqref{eq:rub1_dif0n} and $\hat{r}(0,s,m)$  is  the L.T.~ of ${r}(0,v,m)$ evaluated at $s$.
\end{theorem}
\begin{proof}
	We have that 
	\begin{eqnarray*}
	r(0,0,m) &=& \int_{0}^{\infty}k_n(t)\int_{0}^{ct}r(-ct+x,0,m-1)p(x)dxdt\, \\
	 &=& c^{-1} \int_{0}^{\infty}k_n(y/c) \int_{0}^{y}r(0,y-x,m-1)p(x)dxdy \\
	 &=& \int_{0}^{\infty}\left[ p \ast r(0,y,m-1)\right] \, d  \left[-\bar{K}_n(y/c)\right] \, ,
	\end{eqnarray*}
	noting that $r(x-ct,0,m-1)=r(0,x-ct,0,m-1)$, setting $y=ct$,  and that `$\ast$´ is the convolution sign. Integrating by parts, placing the survival function formula $\bar{K}_n(y/c)$, we get
	\begin{eqnarray*}
		r(0,0,m) &=& \sum_{k=0}^{n-1}\frac{a^k}{k!} \int_{y=0}^{\infty}\, e^{-ay}\, y^k  p\ast r ' (0,y,m-1)dy \\
		&=& \sum_{k=0}^{n-1} \frac{(-a)^k}{k!} \left.  \mathcal{D}_s^{\,k} \widehat{p\ast r' } (0,s,m-1) \right|_{s=a}\\
	& =&	\sum_{k=0}^{n-1} \frac{(-a)^k}{k!} \left.  \mathcal{D}_s^{\,k} s \, \hat{p}(s) \, \hat{r}(0,s,m-1) \right|_{s=a}\,,
	\end{eqnarray*}
using L.T.~ properties.
\hfill
\end{proof}

\section{Examples}
In this section we run some examples for the different problems dealt in previous sections. Some examples we consider are similar to those of \cite{rodriguez2015some}, \cite{afonso2013dividend}, here transposed for  the $Erlang(2,\lambda)$ distributed inter-arrival times. 
%

The particular cases dealt are (`$W-X$'~distributed): $Erlang(2,\lambda)-Erlang(2,\beta)$ and $Erlang(2,\lambda)-$ Combination of exponentials. 
Here, in some cases {we show figures for several values of slope $c$ to understand its impact on final numbers. We consider cases where the income condition \eqref{eq:netprofit} is either satisfied or not. For instance, $c=0.75,\, 1,\, 2.1$. Furthermore, we set $\delta=0.02$, wherever a positive interest force applies}. 
%
%
\subsection{Example Erlang$(2,\lambda)$-Erlang$(2,\beta)$} \ 
%
{Consider first problem dealt in Section~\ref{s:expect_div}, on  the \textbf{Expected Discounted Dividends}. In this particular case we get figures for the expectation by inverting the L.T.~of
	{$\phi_k(u)=\mathbb{E}\left( e^{-\delta \tau_{u}}D_{u}^k \right)$ 
given by, from \eqref{eq_phik_hat},
		$$
		\widehat{\phi}_k(s) = \frac{ (c^2s-2c(\lambda+\delta)) {\tilde{\phi}}(0) + c^2 
			{\tilde{\phi}}'(0) + \lambda^2 p_{{1}} \frac{1-\widehat{p}_{
					{1}}(s)}{s} }
		{(\lambda + \delta - c s)^2 - \lambda^2 \widehat{p}(s)}\,,  
		$$
		%
	{obtaining first $\tilde{\phi}(z)$ and then $\phi(u)=\tilde{\phi}(b-u)$}.
	We show in Tables~\ref{T:ex_v_erlang_c0.75}-\ref{T:ex_v(u,b)_erlang_c2.1} figures for $\mathbb{E}\left( e^{-\delta \tau_{u}}\right)$, $\mathbb{E}\left( e^{-\delta
	\tau_{u}}D_{u}\right)$ and $V(u,b,0.02)$. Similarly as \cite{afonso2013dividend}, in the same tables we also show figures for the probability of getting a dividend, $\chi(u,b)$, although it was calculated from formulae developed in Section \ref{s:single_div}, as the limiting form $\lim_{x \rightarrow \infty}G(u,b;x)=\chi(u,b)$.  We set other parameter values as $\lambda = 2$, $\beta = 1$, as well as a set of several values of $(u,b)$ shown in table captions. }

In Tables~\ref{T:ex_v_erlang_c0.75}-\ref{T:ex_v_erlang_c2.1} we set the initial value equal to the dividend level, i.e., $u=b=1,2,\dots, 10$. Particularly, in Table~\ref{T:ex_v_erlang_c2.1} we exemplify a situation where the income condition is non-standard, we mean reversed. Not surprisingly, for a higher $c$ we have  smaller expectations, correspondingly. Relatively speaking, in each table, the smaller the $b$ the larger is the first expected discounted dividend relative to the whole set of (expected) discounted dividends. 

For the set of Tables~\ref{T:ex_v(u,b)_erlang_c0.75}-\ref{T:ex_v(u,b)_erlang_c2.1} we chose paired values $\{(u,b)=(1,2),\, (1,9),\, (3,6),\, (5,9)\} $, show figures for the corresponding expectations as before and show a column with the product of column~(2) with (3), $\mathbb{E}\left( e^{-0.02 \tau_{u}}\right)  \times  V(b,b,0.02)$, to show the size of the first discounted dividend $\mathbb{E}\left( e^{-0.02
	\tau_{u}}D_{u}\right)$ with the remainder future ones.  The relative contribution is higher for a higher $c$, correspondingly, for a fixed $c$ it decreases when $u$ and $b$  increase.  

For the {\bf{Distribution of a Single Dividend Amount}}, problem dealt in Section~\ref{s:single_div}, we produced some figures for $G(u,b,x)$, for several values of $u$ and $x$, $\{b=5,10\}$, with parameter values {{$\lambda=2$, $\beta=1$}} and $c=1$,  where the income condition is satisfied. We show these figures in
Tables \ref{T:ex_div_b5} and \ref{T:ex_div_b10}, respectively. We recall that $G(0,b,x)=0,\, \ \forall\, b,x$ and that $\lim_{x \rightarrow \infty}G(u,b;x) = \chi(u,b) \,$.

 Figure~\ref{F:G(u,5,x)} 
shows graphs for $G(u,b,x)$ as function of $u$, in the left side for $\{x=1,2,3,4,5\}$, $b=5 $,  and $\{x=1,2,3,5,7,9,10\}$, $b=10 $ in the right side. 

%
%
%
Figure~\ref{F:g(u,5,x)-densi} shows graphs for the corresponding density functions, defective and proper, $g(u,b,x)$ and $\tilde{g}(u,5,x)= g(u,5,x)/\chi(u,5)$, for different values of $u=1,2,\dots,5$, with $b=5$, $\lambda=2$, $\beta=1$ and $c=1$. The left graph shows the defective $g(u,5,x)$ and the right one the conditional $\tilde{g}(u,5,x)$ (this one shows less curves since some are too similar).
%
%
%
Figure~\ref{F:g(u,5,x)-densi-c2.1} shows graphs for corresponding densities, with same parameters except  $c=2.1$, i.e.,
when the income condition is reversed.

On the \textbf{Number of gains prior to ruin} dealt in Section~\ref{s:no_ruin}
we set $c=1$, $\lambda=2$ and $\mu=1$. In Table \ref{T:ex_number_gains} we present the values of the probability of having $m$ gains prior to ruin, $q(u,m)$, for different values of $u$ and $m$ ($m=1,\dots,5$, $u=0.2,0.5,0.7,1,3,5,10$) where $m$ is the number of gains prior to ruin and $u$ the initial surplus.
%
In Figure \ref{F:q(1,m)} we show a graph of the probability function  $q(u,m)$ with  $u=1$. We can observe that it has a slow decreasing tail, and great probability mass concentrated on $m=1,2$ [it is quite similar to those of \cite[Figures~6-7]{egidio2002many}] . 

For the {\bf Number of gains before reaching a given upper target}, or a dividend level, considering the probability of having exactly $m$ gains to reach an upper target $b$, denoted as $r(u,b,m)$ we set the same above parameter values $\{c=1, \lambda=2, \beta=1\}$, and additionally $\{c=0.75, \lambda=1, \beta=1\}$. 
 Tables \ref{T:ex_number_target_b5} and \ref{T:ex_number_target_b10} correspond to first parameter set, Tables~\ref{T:ex2_number_target_b5} and 	\ref{T:ex2_number_target_b10} correspond to the second one. There, we present  values of $r(u,b,m)$ for different values of $u$ and $m$, where $m$ is the number of gains and we have set the upper limit, $b=5$ and $b=10$, respectively.

We can observe that the values from Table~\ref{T:ex_number_target_b5} in the columns corresponding to $u=3$ and $u=5$ and the values in columns corresponding $u=8$ and $u=10$ of Table~\ref{T:ex_number_target_b10} are the same. This is due to the fact that the difference $v=b-u$ is the same, equal to 2 and 0, respectively. Similar remark is applicable to Tables~\ref{T:ex2_number_target_b5} and 	\ref{T:ex2_number_target_b10}. Another remark worth mentioning is that in all Tables~\ref{T:ex_number_target_b5}-\ref{T:ex2_number_target_b10} for $m=1$ and $v=b-u=0$, we have the same figures for that probability. The same happens in Tables~\ref{T:ex_number_target_b5_comb}-\ref{T:ex2_number_target_b10_comb}. This is due to the fact that the integral in \eqref{rq:rub1}, corresponding to 	$r(b,b,1)$,  is independent of $b$.

%
\subsection{Example $Erlang(2,\lambda)-$Combination of exponentials} 
The combination of exponentials is the one that is used in \cite{afonso2013dividend} where 
$$
p(x)=3\, e^{-3\,x/2} -3\,e^{-3\,x}\,, \quad x>0\,.
$$
Like in the previous example we start with the \textbf{Expected Discounted Dividends} problem, but unlike that example we only show figures for $c=0.75$, in Tables~\ref{T:ex_v_comb} and \ref{T:ex_v(u,b)_comb}, since there aren't any other particular aspects that worth to underline. These tables correspond to    
Tables~\ref{T:ex_v_erlang_c0.75} and \ref{T:ex_v(u,b)_erlang_c0.75} in that example, respectively. Corresponding expectations and probabilities are smaller in this example, in most cases.

For the \textbf{Distribution of a Single Dividend Amount}, we show examples with $c=0.75$ and $b=5$. Table~\ref{T:ex_div_b5_comb} shows some figures, graphs of $G(u,5,x)$ as function of $u$ is shown in Figure~\ref{F:G(u,5,x)-comb}, graphs of $g(u,5,x)$ and $\tilde{g}(u,5,x)$ are shown in Figure~\ref{F:g(u,5,x)-densi-comb}. When compared to the corresponding table and graphs of the previous example, they show quite different features.

On the \textbf{Number of gains prior to ruin} dealt in Section~\ref{s:no_ruin} we set $c=1$ and $\lambda=2$.
In Table \ref{T:ex_number_gains_comb} we present the values of the probability of having $m$ gains prior to ruin, $q(u,m)$, for different values of $u$ and $m$ {($m=1,\dots,5\text{; } u=0.2,0.5,0.7,1,3,5,10$)} where $m$ is the number of gains prior to ruin and $u$ the initial surplus.
In Figure \ref{F:q(1,m)-comb} we show a graph of the probability function  $q(u,m)$ with  $u=1$. We can observe that it has a slow decreasing tail, and great probability mass concentrated on $m=1,2$ [{it has a slower decay when compared to Figure~\ref{F:q(1,m)} and those of} \cite[Figures~6-7]{egidio2002many}] . 

For the {\bf Number of gains before reaching a given upper target}, or a dividend level, considering the probability of having exactly $m$ gains to reach an upper target $b$, denoted as $r(u,b,m)$ we set parameter values $\{c=0.75, \lambda=2\}$, and additionally $\{c=0.5, \lambda=2\}$. 
Tables \ref{T:ex_number_target_b5_comb} and \ref{T:ex_number_target_b10_comb} correspond to first parameter set, Tables~\ref{T:ex2_number_target_b5_comb} and 	\ref{T:ex2_number_target_b10_comb} correspond to the second one. There, we present  values of $r(u,b,m)$ for different values of $u$ and $m$, where $m$ is the number of gains and we have set the upper limit, $b=5$ and $b=10$, respectively.

As remarked on the previous example, we can observe that the values from Table~\ref{T:ex_number_target_b5_comb} in the columns corresponding to $u=3$ and $u=5$ and the values in columns corresponding $u=8$ and $u=10$ of Table~\ref{T:ex_number_target_b10_comb} are the same. This is due to the fact that the difference $v=b-u$ is the same, and equal to 2 and 0, respectively. Similar remark is applicable to Tables~\ref{T:ex2_number_target_b5_comb} and 	\ref{T:ex2_number_target_b10_comb}. 

\section*{Acknowledgements}\label{Acknowledgements}
{Authors gratefully acknowledge support from Project CEMAPRE/REM--UIDB/05069/2020 (CEMAPRE/REM-Centre for Applied Mathematics and Economics/Research in Economics and Mathematics) and from Project CMA--UIDB/00297/2020 (Centro de Ma\-te\-má\-ti\-ca e Apli\-ca\-ções) -- financed by FCT/MCTES (Funda\c c\~ao para a Ci\^encia e a Tecnologia/Portuguese Foundation for Science and Technology) through national funds.}

\bibliographystyle{myapalike}
\bibliography{dual_ruin_div}

\bigskip \ \\[1cm]

{\normalsize \noindent {\scriptsize
		\begin{tabular}{lll}
			Renata G.~Alcoforado & {Agnieszka I.~Bergel} & Rui M.R.~Cardoso \\
			& {Department of Management} & {Centro de Matemática e Aplicações (CMA)} \\
			{ISEG and CEMAPRE} & {ISEG and CEMAPRE} & {and Departamento de Matemática} \\
			{Universidade de Lisboa} & {Universidade de Lisboa} & {Faculdade Ciências e Tecnologia}  \\
			{Rua do Quelhas 6} & {Rua do Quelhas 6} & {Universidade Nova de Lisboa} \\
			{1200-781 Lisboa} & {1200-781 Lisboa} & {2829-516 Caparica} \\
			Portugal & {Portugal} & {Portugal}  \\
			orcid.org/0000-0002-7495-8057 &  orcid.org/0000-0002-5153-389X &  orcid.org/0000-0002-0715-7954\\
			alcoforado.renata@gmail.com & {agnieszka@iseg.ulisboa.pt} & rrc@fct.unl.pt \\
			&  &\\
			& & \\
			{Alfredo D.~Egí­dio dos Reis} & {Eugenio V.~Rodriguez-Martinez} & \\
			{Department of Management} & & \\
			{ISEG and CEMAPRE}&  {Fidelidade Seguros \& CEMAPRE}& \\
			{Universidade de Lisboa}& {ISEG, Universidade de Lisboa} & \\
			{Rua do Quelhas 6} & {Rua do Quelhas 6} & \\
			{1200-781 Lisboa}  & {1200-781 Lisboa} & \\
			{Portugal}& {Portugal}  & \\
		    orcid.org/0000-0003-2533-1343	&  orcid.org/0000-0001-7130-3242\\
			{alfredo@iseg.ulisboa.pt}& evrodriguez@gmail.com & \\
		\end{tabular}%
}}

\newpage
\begin{table}
	\centering
	\begin{tabular}{ccccc} \hline \hline \\[-3.5mm]
		\textbf{b}  & \multicolumn{1}{c}{\textbf{$\mathbb{E}\left( e^{-0.02 \tau_{b}}\right)$}} & \multicolumn{1}{c}{\textbf{$\mathbb{E}\left( e^{-0.02
					\tau_{b}}D_{b}\right)$}} & \multicolumn{1}{c}{$V(b,b,0.02)$} & \multicolumn{1}{c}{$\chi(b,b)$} \\ \hline
		\textbf{1}  & 0.66445 & 1.18567 & 3.53353     & 0.67466 \\
		\textbf{2}  & 0.90939 & 1.51352 & 16.70411     & 0.93280 \\
		\textbf{3}  & 0.95824 & 1.57111 & 37.62031     & 0.98707 \\
		\textbf{4}  & 0.96715 & 1.58134 & 48.13442     & 0.99755  \\
		\textbf{5}  & 0.96874 & 1.58316 & 50.64987     & 0.99954 \\
		\textbf{7}  & 0.96908 & 1.58354 & 51.21171     & 0.99998 \\
		\textbf{9}  & 0.96909 & 1.58355 & 51.22976     & 1.00000 \\
		\textbf{10} & 0.96909 & 1.58355 & 51.23025     & 1.00000          \\ \hline \hline              
	\end{tabular}
	\caption{$\mathbb{E}\left( e^{-0.02 \tau_{b}}\right)$, $\mathbb{E}\left( e^{-0.02
			\tau_{b}}D_{b}\right)$, $V(b,b,0.02)$, $\chi(b,b)$; $\lambda = 2$, $\beta = 1$, $c=0.75$}
	\label{T:ex_v_erlang_c0.75}
\end{table}
\begin{table}
	\centering
	\begin{tabular}{ccccc} \hline \hline \\[-3.5mm]
		\textbf{b}  & \multicolumn{1}{c}{\textbf{$\mathbb{E}\left( e^{-0.02 \tau_{b}}\right)$}} & \multicolumn{1}{c}{\textbf{$\mathbb{E}\left( e^{-0.02
					\tau_{b}}D_{b}\right)$}} & \multicolumn{1}{c}{$V(b,b,0.02)$} & \multicolumn{1}{c}{$\chi(b,b)$} \\ \hline
		\textbf{1}  & 0.50218                        & 0.91481                          & 1.83765                        & 0.50822                         \\
		\textbf{2}  & 0.79762                        & 1.32791                           & 6.56131                        & 0.81577                         \\
		\textbf{3}  & 0.90491                         & 1.45485                           & 15.29962                        & 0.93251                         \\
		\textbf{4}  & 0.94222                        & 1.49649                           & 25.89859                        & 0.97527                         \\
		\textbf{5}  & 0.95513                        & 1.51064                           & 33.66764                        & 0.99092                         \\
		\textbf{7}  & 0.96116                        & 1.51719                           & 39.05850                        & 0.99877                         \\
		\textbf{9}  & 0.96188                        & 1.51798                           & 39.82178                        & 0.99983                         \\
		\textbf{10} & 0.96194                        & 1.51805                           & 39.89141                        & 0.99994          \\ \hline \hline              
	\end{tabular}
	\caption{$\mathbb{E}\left( e^{-0.02 \tau_{b}}\right)$, $\mathbb{E}\left( e^{-0.02
			\tau_{b}}D_{b}\right)$, $V(b,b,0.02)$, $\chi(b,b)$; $\lambda = 2$, $\beta = 1$, $c=1$}
	\label{T:ex_v_erlang_c1}
\end{table}
\begin{table}
	\centering
	\begin{tabular}{ccccc} \hline \hline \\[-3.5mm]
		\textbf{b}  & \multicolumn{1}{c}{\textbf{$\mathbb{E}\left( e^{-0.02 \tau_{b}}\right)$}} & \multicolumn{1}{c}{\textbf{$\mathbb{E}\left( e^{-0.02
					\tau_{b}}D_{b}\right)$}} & \multicolumn{1}{c}{$V(b,b,0.02)$} & \multicolumn{1}{c}{$\chi(b,b)$} \\ \hline
		\textbf{1}  & 0.18569                        & 0.35710                          & 0.43853                       & 0.18677                         \\
		\textbf{2}  & 0.37520                        & 0.67356                          & 1.07805                        & 0.37956                         \\
		\textbf{3}  & 0.49632                        & 0.84771                          & 1.68302                        & 0.50504                         \\
		\textbf{4}  & 0.57398                        & 0.94912                          & 2.22789                        & 0.58756                         \\
		\textbf{5}  & 0.62610                        & 1.01327                           & 2.71002                        & 0.64473                         \\
		\textbf{7}  & 0.68915                        & 1.08716                           & 3.49733                        & 0.71788                         \\
		\textbf{9}  & 0.72393                        & 1.12633                           & 4.07993                        & 0.76240                         \\
		\textbf{10} & 0.73554                         & 1.13917                           & 4.30751                        & 0.77865          \\ \hline \hline              
	\end{tabular}
	\caption{$\mathbb{E}\left( e^{-0.02 \tau_{b}}\right)$, $\mathbb{E}\left( e^{-0.02
			\tau_{b}}D_{b}\right)$, $V(b,b,0.02)$, $\chi(b,b)$; $\lambda = 2$, $\beta = 1$, $c=2.1$}
	\label{T:ex_v_erlang_c2.1}
\end{table}
\begin{landscape}
	\begin{table}[]
	\centering
	\begin{tabular}{ccccccc}  \hline \hline \\[-3.5mm]
		$(u,b)$ & $\mathbb{E}\left( e^{-0.02
			\tau_{u}}D_{u}\right)$  & $\mathbb{E}\left( e^{-0.02 \tau_{u}}\right)$ (2)   & $V(b,b,0.02)$ (3) & (2) $\times$ (3)     & $V(u,b,0.02)$  & $\chi(u,b)$ \\  \hline 
		$(1,2)$ & 0.94704                    & 0.63804                  & 16.70411                         & 10.65792                 & 11.60496                     & 0.65587                     \\
		$(1,9)$ & 0.81080                    & 0.57189                  & 51.22976                         & 29.29770                 & 30.10850                    & 0.65537                     \\
		$(3,6)$ & 1.27723                     & 0.91397                  & 51.12592                         & 46.72777                 & 48.00500                     & 0.98639                      \\
		$(5,9)$ & 1.27042                     & 0.91014                  & 51.22976                         & 46.62648                 & 47.89690                    & 0.99951                    
		 \\  \hline \hline
	\end{tabular}
	\caption{$\mathbb{E}\left( e^{-0.02
			\tau_{u}}D_{u}\right)$, $\mathbb{E}\left( e^{-0.02 \tau_{u}}\right)$, $V(b,b,0.02)$, $V(u,b,0.02)$, $\chi(u,b)$; $\lambda = 2$, $\beta = 1$, $c=0.75$}
	\label{T:ex_v(u,b)_erlang_c0.75}
\end{table}
\begin{table}[]
	\centering
	\begin{tabular}{ccccccc}  \hline \hline \\[-3.5mm]
		$(u,b)$ & $\mathbb{E}\left( e^{-0.02
			\tau_{u}}D_{u}\right)$  & $\mathbb{E}\left( e^{-0.02 \tau_{u}}\right)$ (2)   & $V(b,b,0.02)$ (3) & (2) $\times$ (3)     & $V(u,b,0.02)$  & $\chi(u,b)$ \\  \hline 
		$(1,2)$ & 0.68765 & 0.45254 & 6.56131     & 2.96927 & 3.65691 & 0.46254 \\
		$(1,9)$ & 0.52816 & 0.38404 & 39.82178     & 15.29300  & 15.82116 & 0.44598 \\
		$(3,6)$ & 1.16429  & 0.84462 & 37.51689     & 31.68760 & 32.85189 & 0.91918 \\
		$(5,9)$ & 1.21507  & 0.88320 & 39.82178     & 35.17081 & 36.38588 & 0.98898 \\  \hline \hline
	\end{tabular}
	\caption{$\mathbb{E}\left( e^{-0.02
			\tau_{u}}D_{u}\right)$, $\mathbb{E}\left( e^{-0.02 \tau_{u}}\right)$, $V(b,b,0.02)$, $V(u,b,0.02)$, $\chi(u,b)$; $\lambda = 2$, $\beta = 1$, $c=1$}
\label{T:ex_v(u,b)_erlang_c1}
\end{table}
\begin{table}[]
	\centering
	\begin{tabular}{ccccccc}  \hline \hline \\[-3.5mm]
		$(u,b)$ & $\mathbb{E}\left( e^{-0.02
			\tau_{u}}D_{u}\right)$  & $\mathbb{E}\left( e^{-0.02 \tau_{u}}\right)$ (2)   & $V(b,b,0.02)$ (3) & (2) $\times$ (3)     & $V(u,b,0.02)$  & $\chi(u,b)$ \\  \hline 
		$(1,2)$ & 0.22002                    & 0.12751                  & 1.07805                         & 0.13747                & 0.35748                   & 0.12863                     \\
		$(1,9)$ & 0.04402                   & 0.03084                 & 4.07993                         & 0.12582                & 0.16984                   & 0.03465                    \\
		$(3,6)$ & 0.38458                    & 0.26199                  & 3.13178                         & 0.82050                & 1.20507                    & 0.27667                     \\
		$(5,9)$ & 0.45775                     & 0.32041                  & 4.07993                         & 1.30723                 & 1.76498                    & 0.35670                    
		 \\  \hline \hline
	\end{tabular} 
	\caption{$\mathbb{E}\left( e^{-0.02
			\tau_{u}}D_{u}\right)$, $\mathbb{E}\left( e^{-0.02 \tau_{u}}\right)$, $V(b,b,0.02)$, $V(u,b,0.02)$, $\chi(u,b)$; $\lambda = 2$, $\beta = 1$, $c=2.1$}
	\label{T:ex_v(u,b)_erlang_c2.1}
\end{table}
\end{landscape}
%
\begin{table}[h]
	\centering
	\begin{tabular}{clllll} \hline \hline
		\textbf{u\textbackslash{}x} & \multicolumn{1}{c}{\textbf{1}} & \multicolumn{1}{c}{\textbf{2}} & \multicolumn{1}{c}{\textbf{3}} & \multicolumn{1}{c}{\textbf{4}} & \multicolumn{1}{c}{\textbf{5}} \\ \hline
		\textbf{1}                  & 0.185349                       & 0.31486                        & 0.385064                       & 0.41919                        & 0.41919                        \\
		\textbf{2}                  & 0.324635                       & 0.551758                       & 0.674932                       & 0.73482                        & 0.762214                       \\
		\textbf{3}                  & 0.379865                       & 0.647155                       & 0.792408                       & 0.863104                       & 0.895462                       \\
		\textbf{4}                  & 0.393353                       & 0.677656                       & 0.8336                         & 0.909861                       & 0.944866                       \\
		\textbf{5}                  & 0.363765                       & 0.663592                       & 0.834962                       & 0.920472                       & 0.960195           \\ \hline \hline           
	\end{tabular}
	\caption{Distribution $G(u,5,x)$, $b=5$; $\lambda=2$, $\beta=1$, $c=1$}
	\label{T:ex_div_b5}
\end{table}
\begin{table}[h]
	\centering
	\begin{tabular}{clllllll} \hline \hline 
		\textbf{u\textbackslash{}x} & \multicolumn{1}{c}{\textbf{1}} & \multicolumn{1}{c}{\textbf{2}} & \multicolumn{1}{c}{\textbf{3}} & \multicolumn{1}{c}{\textbf{5}} & \multicolumn{1}{c}{\textbf{7}} & \multicolumn{1}{c}{\textbf{9}} & \multicolumn{1}{c}{\textbf{10}} \\ \hline
		\textbf{1}                  & 0.185125                       & 0.314405                       & 0.384471                       & 0.4341                         & 0.443882                       & 0.445621                       & 0.445828                        \\
		\textbf{2}                  & 0.32454                        & 0.55118                        & 0.674011                       & 0.761016                       & 0.778164                       & 0.781213                       & 0.781576                        \\
		\textbf{3}                  & 0.381331                       & 0.647632                       & 0.791958                       & 0.894188                       & 0.914338                       & 0.917919                       & 0.918347                        \\
		\textbf{4}                  & 0.402645                       & 0.683834                       & 0.83623                        & 0.944175                       & 0.965452                       & 0.969233                       & 0.969685                        \\
		\textbf{5}                  & 0.410505                       & 0.697198                       & 0.852579                       & 0.96264                        & 0.984334                       & 0.98819                        & 0.98865                         \\
		\textbf{6}                  & 0.413332                       & 0.702069                       & 0.858571                       & 0.969433                       & 0.991285                       & 0.995169                       & 0.995633                        \\
		\textbf{7}                  & 0.414052                       & 0.703625                       & 0.860646                       & 0.971903                       & 0.993837                       & 0.997736                       & 0.998202                        \\
		\textbf{8}                  & 0.412789                       & 0.703074                       & 0.860789                       & 0.972672                       & 0.99475                        & 0.998676                       & 0.999145                        \\
		\textbf{9}                  & 0.405047                       & 0.697517                       & 0.857887                       & 0.972288                       & 0.994959                       & 0.999                          & 0.999484                        \\
		\textbf{10}                 & 0.367508                       & 0.669949                       & 0.842736                       & 0.968972                       & 0.994436                       & 0.999017                       & 0.999567                      \\ \hline \hline 
	\end{tabular}
	\caption{Distribution $G(u,5,x)$, $b=10$; ; $\lambda=2$, $\beta=1$, $c=1$}
	\label{T:ex_div_b10}
\end{table}
%
\begin{figure}[h]
	\centering
	\includegraphics[width=7.5cm]{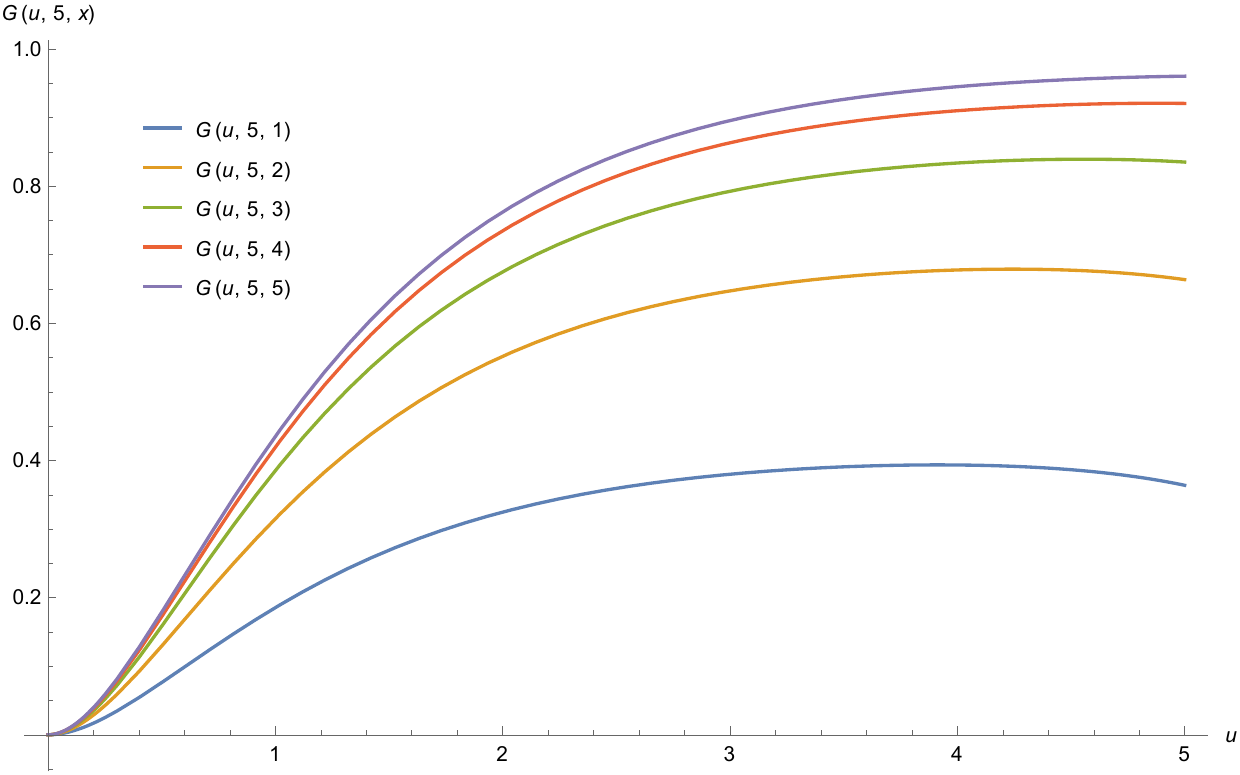}
	\includegraphics[width=7.5cm]{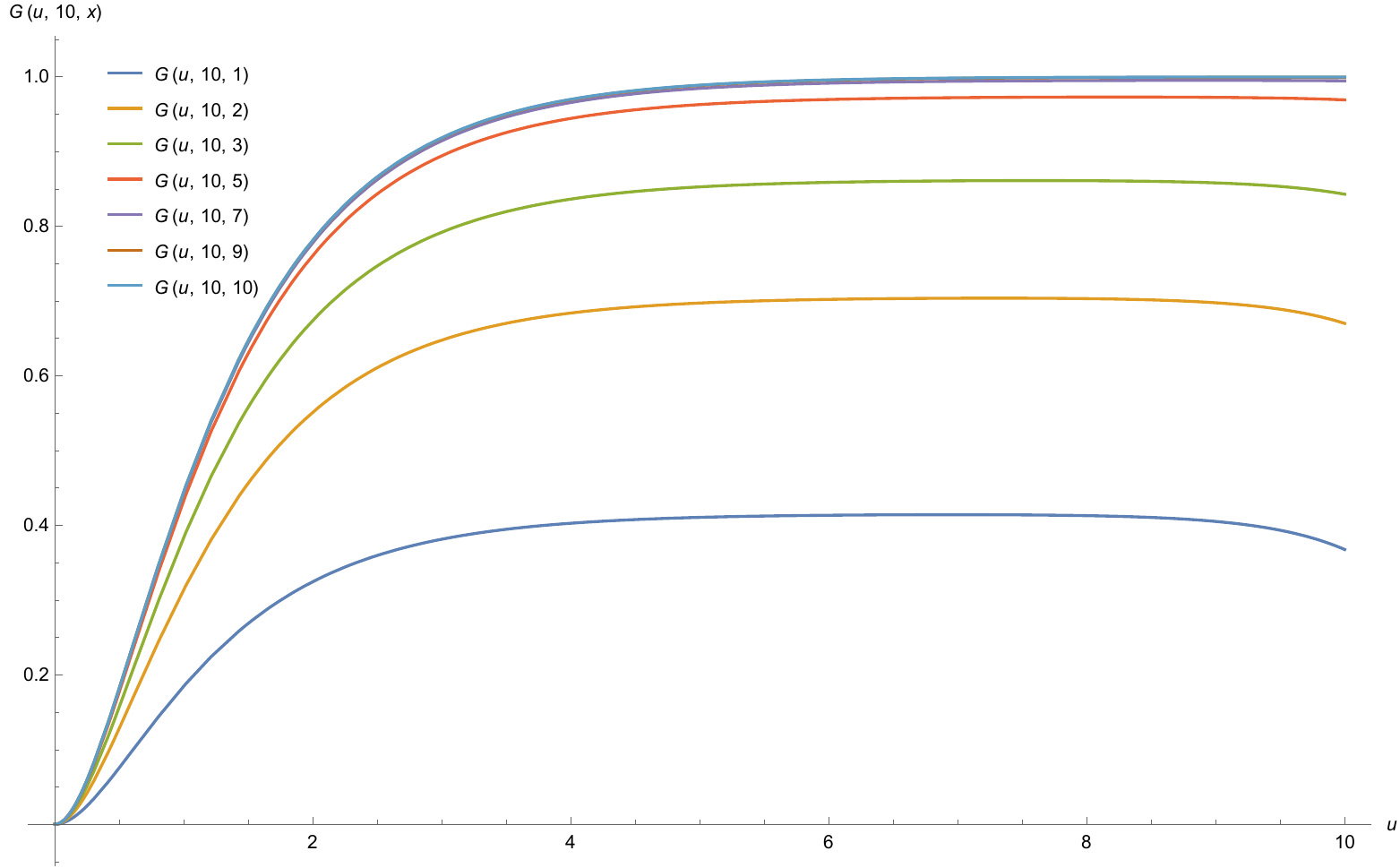}
	\caption{$G(u,b,x)$ as function of $u$, $\lambda=2$, $\beta=1$, $c=1$, $\{x=1,\dots,5;\, b=5\}$; $\{x=1,\dots,10;\,b=10\}$}
	\label{F:G(u,5,x)}
\end{figure}
\begin{figure}[h]
	\centering
	\centering
	\includegraphics[width=.475\linewidth, height=150px]{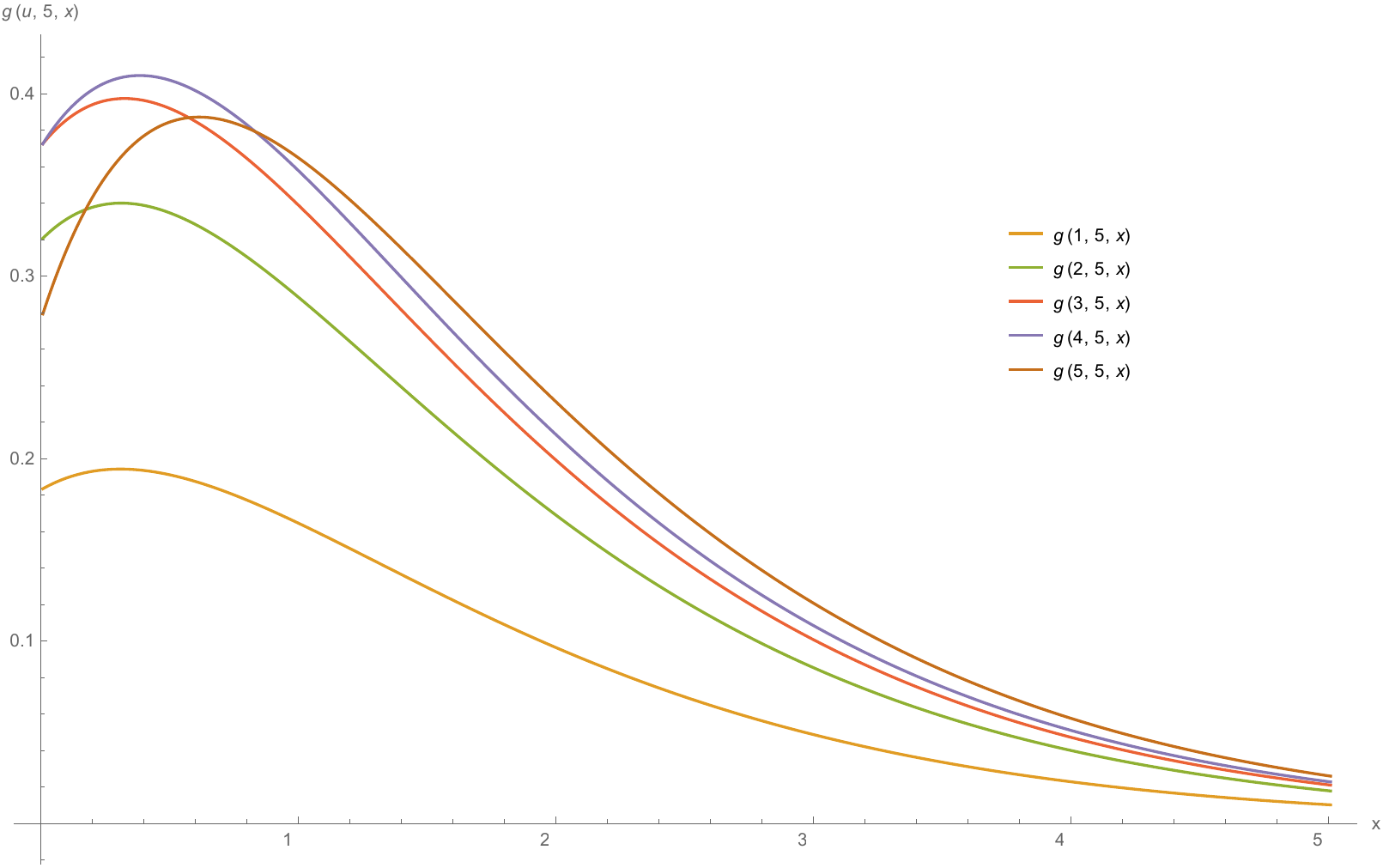}
	\includegraphics[width=.475\linewidth, height=150px]{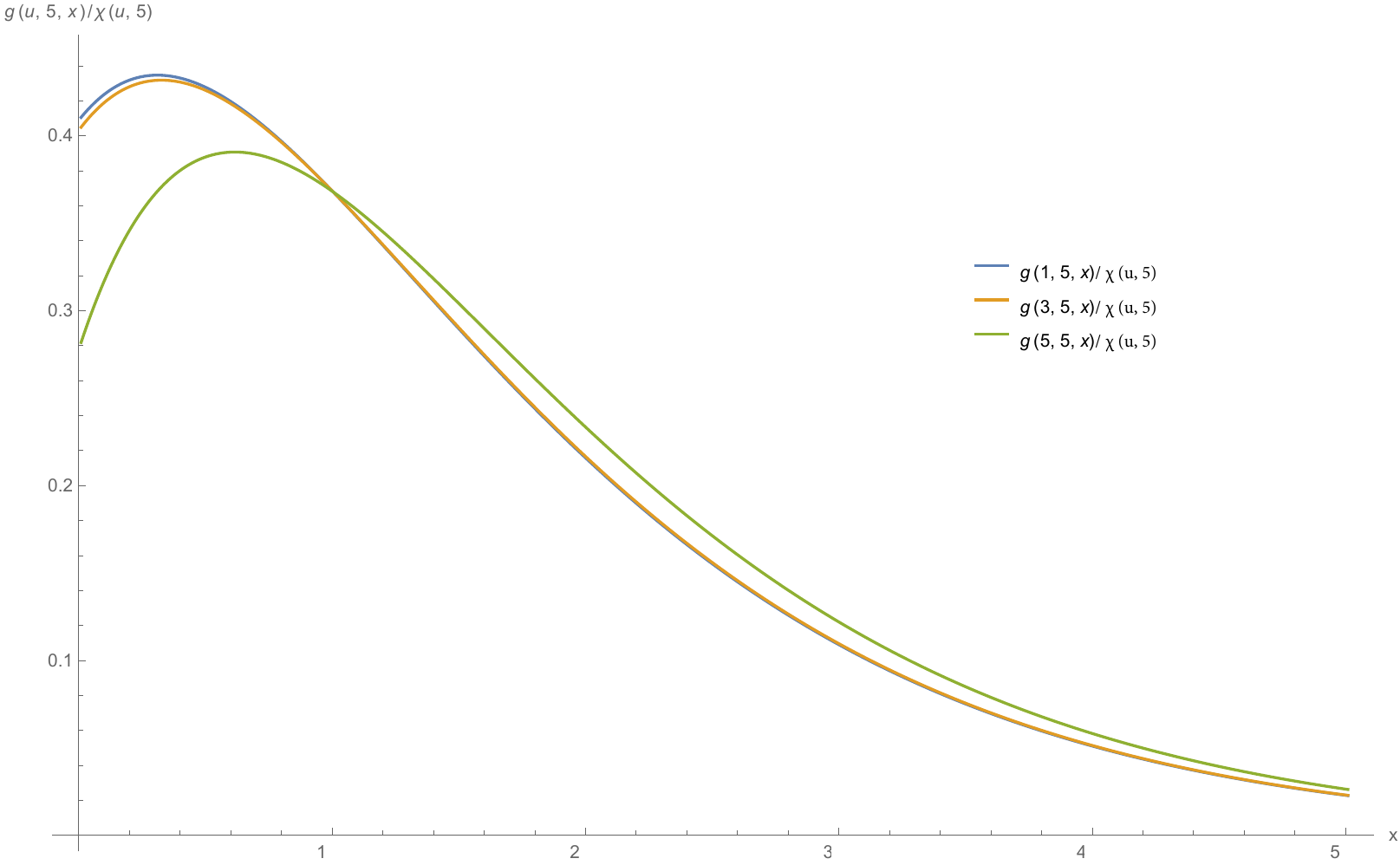}
	\caption{Densities $g(u,5,x)$ and $\tilde{g}(u,5,x)$, $u=1,2,\dots,5$, $b=5$; $\lambda=2$, $\beta=1$, $c=1$}
	\label{F:g(u,5,x)-densi}
\end{figure}

\begin{figure}[h]
	\centering
	\centering
	\includegraphics[width=.475\linewidth, height=150px]{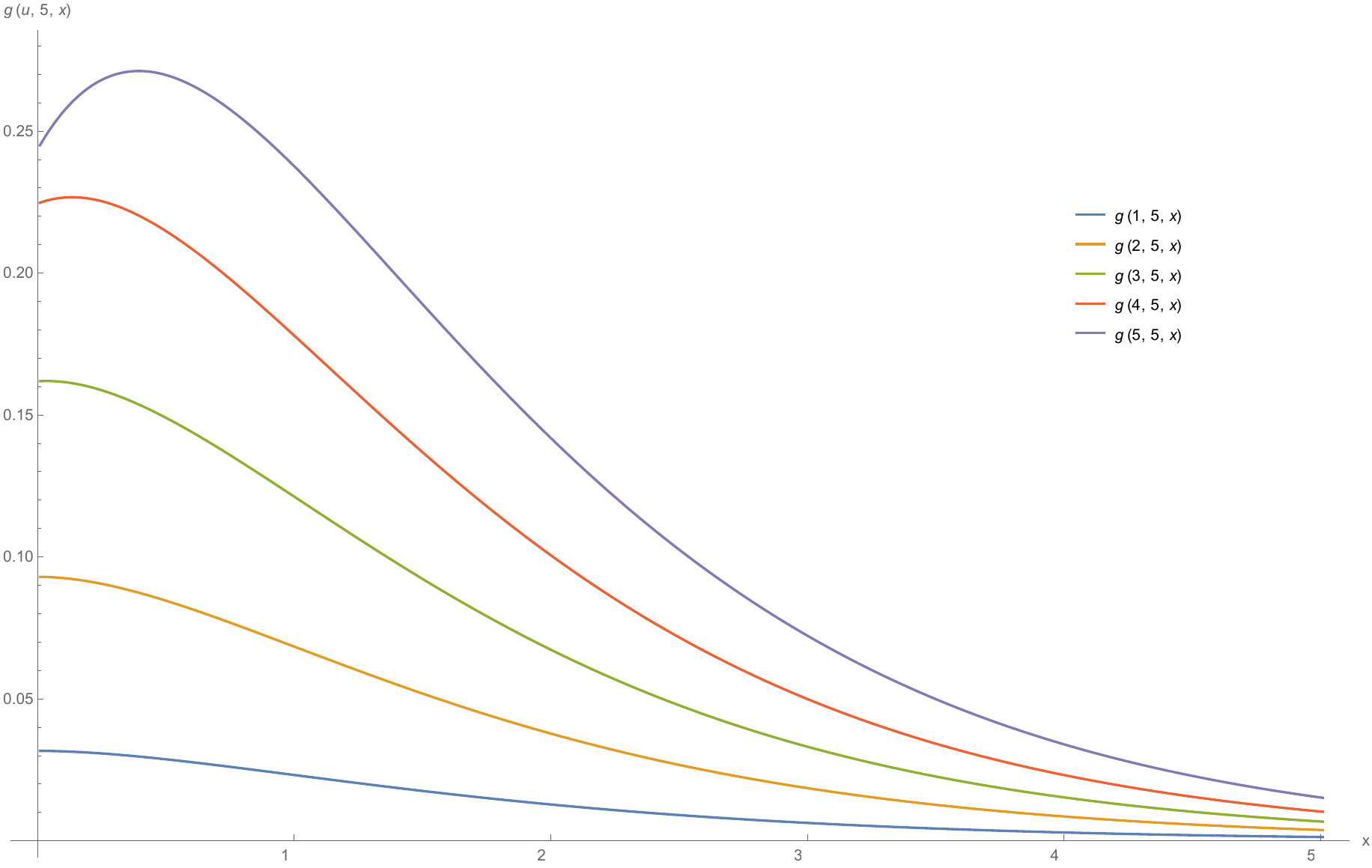}
	\includegraphics[width=.475\linewidth, height=150px]{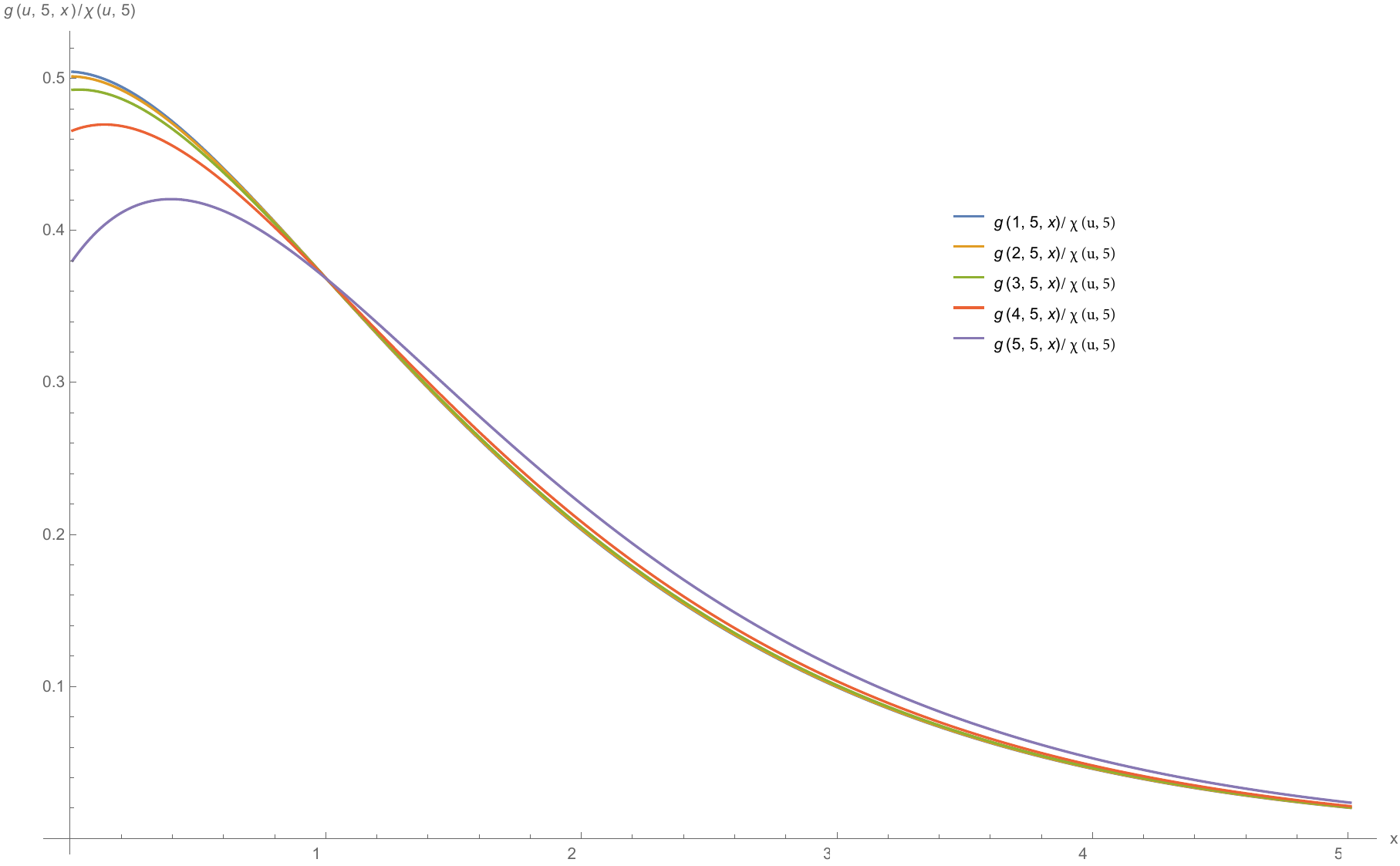}
	\caption{Densities $g(u,5,x)$ and $\tilde{g}(u,5,x)$, $u=1,2,\dots,5$, $b=5$; $\lambda=2$, $\beta=1$, $c=2.1$}
	\label{F:g(u,5,x)-densi-c2.1}
\end{figure}

\clearpage
\begin{table}[h]
	\centering
	\begin{tabular}{lllllllll}
		\hline \hline
		\textbf{$m$\textbackslash{}$u$} & \textbf{0} & \textbf{0.2} & \textbf{0.5} & \textbf{0.7} & \textbf{1} & \textbf{3} & \textbf{5}  & \textbf{10}                    \\ \hline
		\textbf{0}                  & 1          & 0.938448     & 0.735759     & 0.591833     & 0.406006   & 0.017351  & 0.0004994 & $4.33\times10^{-8}$ \\
		\textbf{1}                  & 0          & 0.014697    & 0.054501    & 0.075185    & 0.090224  & 0.021482  & 0.0014293  & $4.12\times10^{-7}$ \\
		\textbf{2}                  & 0          & 0.003270   & 0.013852    & 0.020929    & 0.028787  & 0.015305   & 0.0018590  & $1.47\times10^{-6}$ \\
		\textbf{3}                  & 0          & 0.001317   & 0.005707   & 0.008798   & 0.012558  & 0.009655 & 0.0017568  & $3.16\times10^{-6}$ \\
		\textbf{4}                  & 0          & 0.000655  & 0.002866   & 0.004454   & 0.006455 & 0.006024 & 0.0014366  & $4.97\times10^{-6}$ \\
		\textbf{5}                  & 0          & 0.000364  & 0.001600   & 0.002496   & 0.003648 & 0.003820 & 0.0010973  & $6.39\times10^{-6}$
		\\ \hline \hline
	\end{tabular}
	\caption{Probability of having $m$ gains prior to ruin, $q(u, m)$ }
	\label{T:ex_number_gains}
\end{table}
%
%
%
\begin{figure}[h]
	\centering
	\includegraphics[width=12cm]{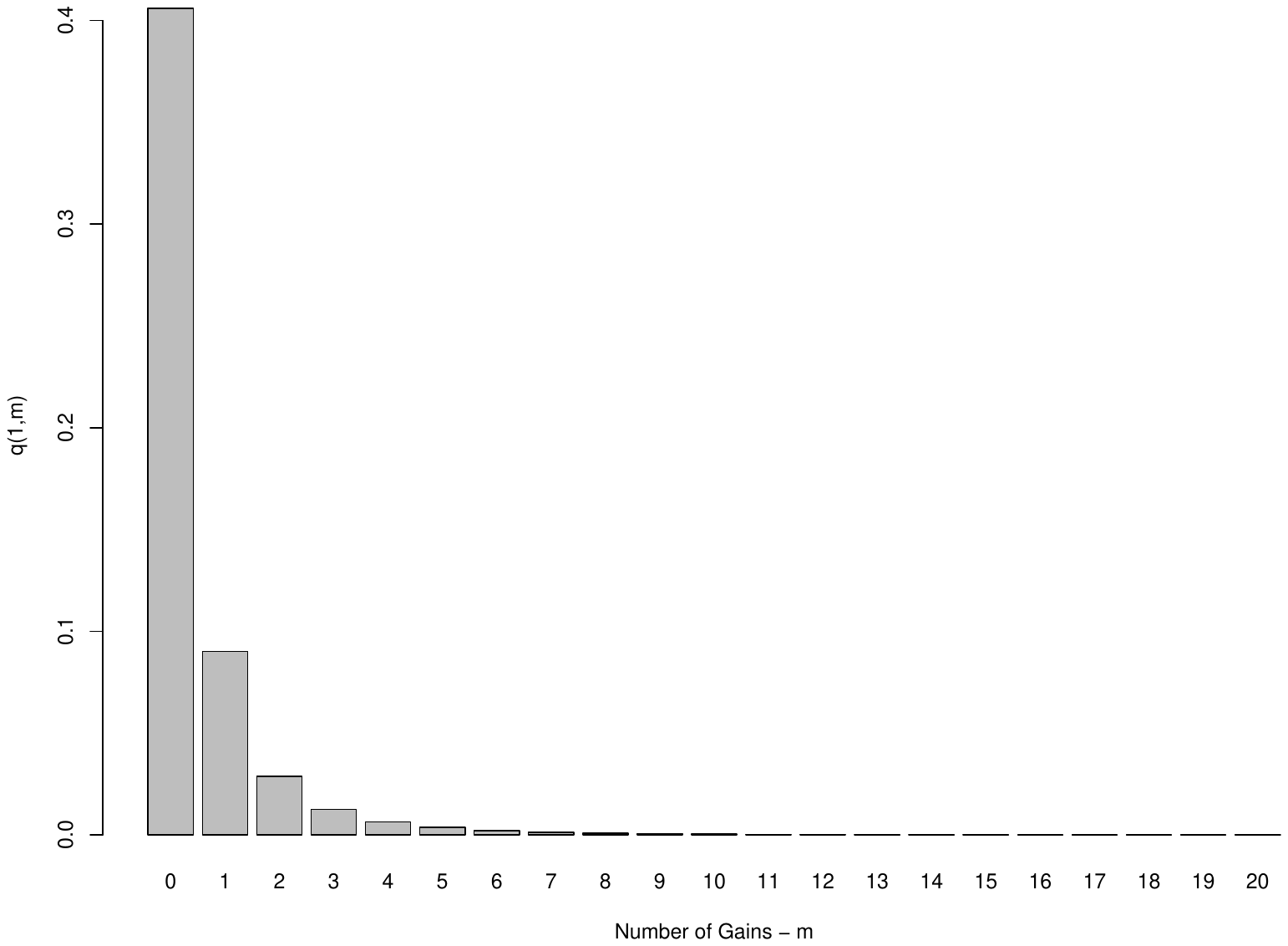}
	\caption{Probability of having $m$ gains prior to ruin, $q(1, m)$}
	\label{F:q(1,m)}
\end{figure}

\clearpage
%
\begin{landscape}
	\begin{table}
	\centering
	\begin{tabular}{llllllllll}
		\hline \hline
		\multicolumn{1}{c}{\textbf{m\textbackslash{}u}} & \multicolumn{1}{c}{\textbf{0}} & \multicolumn{1}{c}{\textbf{0.2}} & \multicolumn{1}{c}{\textbf{0.5}} & \multicolumn{1}{c}{\textbf{0.7}} & \multicolumn{1}{c}{\textbf{1}} & \multicolumn{1}{c}{\textbf{2}} & \multicolumn{1}{c}{\textbf{3}} & \multicolumn{1}{c}{\textbf{4}} & \multicolumn{1}{c}{\textbf{5}}\\ \hline
		\textbf{1}                                      & 0.01996                              & 0.02365                        & 0.03045                        & 0.03598                        & 0.04613                      & 0.10326                       & 0.22055                       & 0.43600                & 0.74074       \\
		\textbf{2}                                      & 0.08858                              & 0.09870                        & 0.11543                         & 0.12760                         & 0.14729                       & 0.22093                        & 0.28195                       & 0.26824              & 0.13900         \\
		\textbf{$\geq$3}                       & 0.89146                              & 0.87765                        & 0.85412                        & 0.83642                        & 0.80658                      & 0.67581                       & 0.49750                       & 0.29576               &      0.12026   \\ \hline \hline
	\end{tabular}
	\caption{Probability of having $m$ gains prior to reach an upper target $b=5$, $r(u,5,m)$}
	\label{T:ex_number_target_b5}
\end{table}
\begin{table}
	\centering
	\begin{tabular}{llllllllll}
		\hline \hline
		\multicolumn{1}{c}{\textbf{m\textbackslash{}u}} & \multicolumn{1}{c}{\textbf{0}} & \multicolumn{1}{c}{\textbf{0.2}} & \multicolumn{1}{c}{\textbf{0.5}} & \multicolumn{1}{c}{\textbf{0.7}} & \multicolumn{1}{c}{\textbf{1}} & \multicolumn{1}{c}{\textbf{3}} & \multicolumn{1}{c}{\textbf{5}} & \multicolumn{1}{c}{\textbf{8}} & \multicolumn{1}{c}{\textbf{10}}\\ \hline
		\textbf{1}                                      & 0.00024                              & 0.00028                      & 0.00037                      & 0.00044                      & 0.00058                    & 0.00351                     & 0.01996                      & 0.22055             &       0.74074   \\
		\textbf{2}                                      & 0.00317                              & 0.00368                       & 0.00459                       & 0.00532                      & 0.00661                     & 0.02631                      & 0.08858                      & 0.28195              &0.13900         \\
		\textbf{$\geq$3}                       & 0.99659                              & 0.99603                      & 0.99503                      & 0.99424                      & 0.99280                    & 0.97018                    & 0.89146                      & 0.49750             &0.12026         \\ \hline \hline
	\end{tabular}
	\caption{Probability of having $m$ gains prior to reach an upper target $b=10$, $r(u,10,m)$}
	\label{T:ex_number_target_b10}
\end{table}
\begin{table}
		\centering
	\begin{tabular}{clllllllll}
		\hline \hline
		\textbf{m\textbackslash{}u} & \multicolumn{1}{c}{\textbf{0}} & \multicolumn{1}{c}{\textbf{0.2}} & \multicolumn{1}{c}{\textbf{0.5}} & \multicolumn{1}{c}{\textbf{0.7}} & \multicolumn{1}{c}{\textbf{1}} & \multicolumn{1}{c}{\textbf{2}} & \multicolumn{1}{c}{\textbf{3}} & \multicolumn{1}{c}{\textbf{4}} & \multicolumn{1}{c}{\textbf{5}} \\ \hline
		\textbf{1}                  & 0.01996                        & 0.02365                        & 0.03045                        & 0.03598                        & 0.0461                         & 0.10326                        & 0.22055                        & 0.43601                        & 0.74074                        \\
		\textbf{2}                  & 0.08858                        & 0.09870                        & 0.11543                         & 0.12760                         & 0.14729                        & 0.22093                        & 0.28195                        & 0.26825                        & 0.14268                              \\
		\textbf{$\geq$3}   & 0.89146                        & 0.87765                        & 0.85412                        & 0.83642                        & 0.80661                        & 0.67581                        & 0.4975                         & 0.29574                        & 0.11658      \\ \hline \hline                 
	\end{tabular}
	\caption{Probability of having $m$ gains prior to reach an upper target $b=5$, $r(u,5,m)$}
	\label{T:ex2_number_target_b5}
\end{table}
\begin{table}
		\centering
	\begin{tabular}{clllllllll}
		\hline \hline
		\textbf{m\textbackslash{}u} & \multicolumn{1}{c}{\textbf{0}} & \multicolumn{1}{c}{\textbf{0.2}} & \multicolumn{1}{c}{\textbf{0.7}} & \multicolumn{1}{c}{\textbf{1}} & \multicolumn{1}{c}{\textbf{3}} & \multicolumn{1}{c}{\textbf{5}} & \multicolumn{1}{c}{\textbf{8}} & \multicolumn{1}{c}{\textbf{9}} & \multicolumn{1}{c}{\textbf{10}} \\ \hline
		\textbf{1}                  & 0.00024                     & 0.00028                      & 0.00044                      & 0.00058                     & 0.00351                      & 0.01996                       & 0.22055                        & 0.43601                        & 0.74074                         \\
		\textbf{2}                  & 0.00317                      & 0.00368                       & 0.00532                       & 0.00661                      & 0.02631                       & 0.08858                       & 0.28195                        & 0.26825                        & 0.14268                               \\
		\textbf{$\geq$3}   & 0.99659                     & 0.99604                      & 0.99424                      & 0.99281                     & 0.97018                      & 0.89146                        & 0.4975                         & 0.29574                        & 0.11658      \\ \hline \hline                 
	\end{tabular}
	\caption{Probability of having $m$ gains prior to reach an upper target $b=10$, $r(u,10,m)$}
	\label{T:ex2_number_target_b10}
\end{table}
\end{landscape}

\clearpage
\begin{landscape}
\begin{table}
	\centering
	\begin{tabular}{ccccc} \hline \hline \\[-3.5mm]
		\textbf{b}  & \multicolumn{1}{c}{\textbf{$\mathbb{E}\left( e^{-0.02 \tau_{b}}\right)$}} & \multicolumn{1}{c}{\textbf{$\mathbb{E}\left( e^{-0.02
					\tau_{b}}D_{b}\right)$}} & \multicolumn{1}{c}{$V(b,b,0.02)$} & \multicolumn{1}{c}{$\chi(b,b)$} \\ \hline
		\textbf{1}  & 0.62876                        & 0.49203                          & 1.32539                        & 0.65271                         \\
		\textbf{2}  & 0.81787                        & 0.64454                          & 3.53884                        & 0.85508                          \\
		\textbf{3}  & 0.88307                        & 0.69335                          & 5.92938                        & 0.93135                         \\
		\textbf{4}  & 0.90917                        & 0.71233                          & 7.84241                        & 0.96564                         \\
		\textbf{5}  & 0.92029                        & 0.72032                          & 9.03706                        & 0.98235                          \\
		\textbf{7}  & 0.92731                        & 0.72534                          & 9.97898                        & 0.99519                         \\
		\textbf{9}  & 0.92870                        & 0.72629                          & 10.18618                        & 0.99867                         \\
		\textbf{10} & 0.92889                        & 0.72570                          & 10.20519                        & 0.99930          \\ \hline \hline              
	\end{tabular}
	\caption{$\mathbb{E}\left( e^{-0.02 \tau_{b}}\right)$, $\mathbb{E}\left( e^{-0.02
			\tau_{b}}D_{b}\right)$, $V(b,b,0.02)$, $\chi(b,b)$, Combination of Exponentials; $\lambda = 2$, $c=0.75$}
	\label{T:ex_v_comb}
\end{table}
\begin{table}[]
	\centering
	\begin{tabular}{ccccccc}  \hline \hline \\[-3.5mm]
		$(u,b)$ & $\mathbb{E}\left( e^{-0.02
			\tau_{u}}D_{u}\right)$  & $\mathbb{E}\left( e^{-0.02 \tau_{u}}\right)$ (2)   & $V(b,b,0.02)$ (3) & (2) $\times$ (3)     & $V(u,b,0.02)$  & $\chi(u,b)$ \\  \hline 
		$(1,2)$ & 0.32788                    & 0.45990                    & 3.82068                         & 1.75713                 & 2.08501                    & 0.47518                     \\
		$(1,9)$ & 0.13838                    & 0.23844                  & 13.60174                         & 3.24321                 & 3.38159                    & 0.36637                     \\
		$(3,6)$ & 0.48838                     & 0.68862                  & 12.70810                         & 8.75111                 & 9.23949                    & 0.83678                     \\
		$(5,9)$ & 0.49763                    & 0.70435                  & 13.60174                         & 9.58044                 & 10.07806                    & 0.95289                 \\ \hline \hline
	\end{tabular}
	\caption{$\mathbb{E}\left( e^{-0.02
			\tau_{u}}D_{u}\right)$, $\mathbb{E}\left( e^{-0.02 \tau_{u}}\right)$, $V(b,b,0.02)$, $V(u,b,0.02)$, $\chi(u,b)$, Combination of Exponentials; $\lambda = 2$, $c=0.75$}
	\label{T:ex_v(u,b)_comb}
\end{table}
\end{landscape}
\clearpage
\begin{table}
	\centering
	\begin{tabular}{clllll} \hline \hline
		\textbf{u\textbackslash{}x} & \multicolumn{1}{c}{\textbf{1}} & \multicolumn{1}{c}{\textbf{2}} & \multicolumn{1}{c}{\textbf{3}} & \multicolumn{1}{c}{\textbf{4}} & \multicolumn{1}{c}{\textbf{5}} \\ \hline
		\textbf{1}                  & 0.171138                       & 0.325939                       & 0.366285                       & 0.375577                       & 0.377664                       \\
		\textbf{2}                  & 0.310213                       & 0.590809                       & 0.663943                       & 0.680785                       & 0.684569                       \\
		\textbf{3}                  & 0.384908                       & 0.733022                       & 0.823752                       & 0.844647                       & 0.849342                       \\
		\textbf{4}                  & 0.425217                       & 0.808336                       & 0.908172                       & 0.931163                       & 0.936328                       \\
		\textbf{5}                  & 0.473438                       & 0.854693                       & 0.953485                       & 0.976211                       & 0.981316      \\ \hline \hline                
	\end{tabular}
	\caption{Distribution $G(u,5,x)$, $b=5$, $c=0.75$}
	\label{T:ex_div_b5_comb}
\end{table}
\begin{figure}
	\centering
	\includegraphics[width=7.5cm]{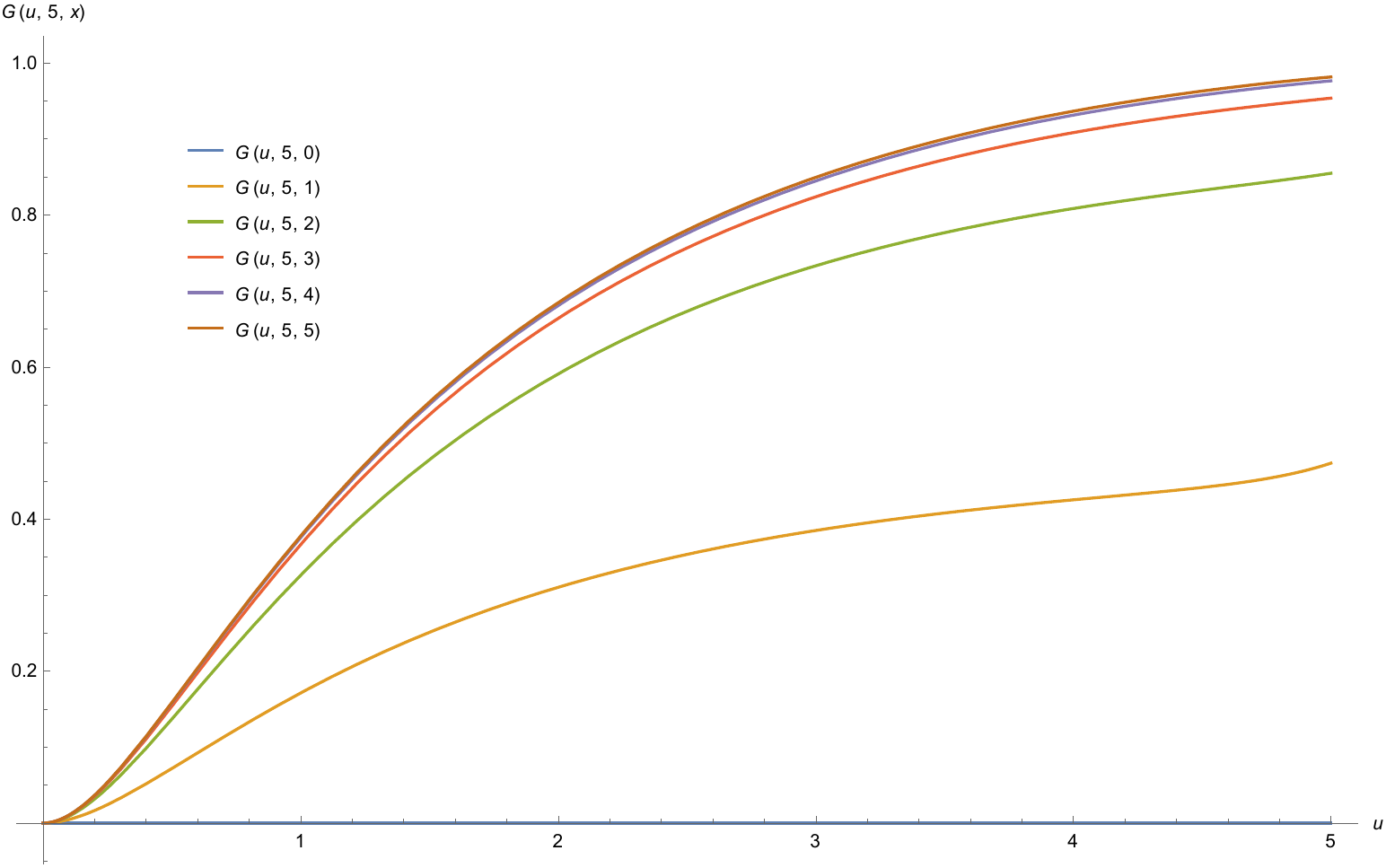}
	\caption{$G(u,5,x)$ as function of $u$, $x=0,1,2,3,4,5$, $c=0.75$ }
	\label{F:G(u,5,x)-comb}
\end{figure}
\begin{figure}
	\centering
	\centering
	\includegraphics[width=.475\linewidth, height=150px]{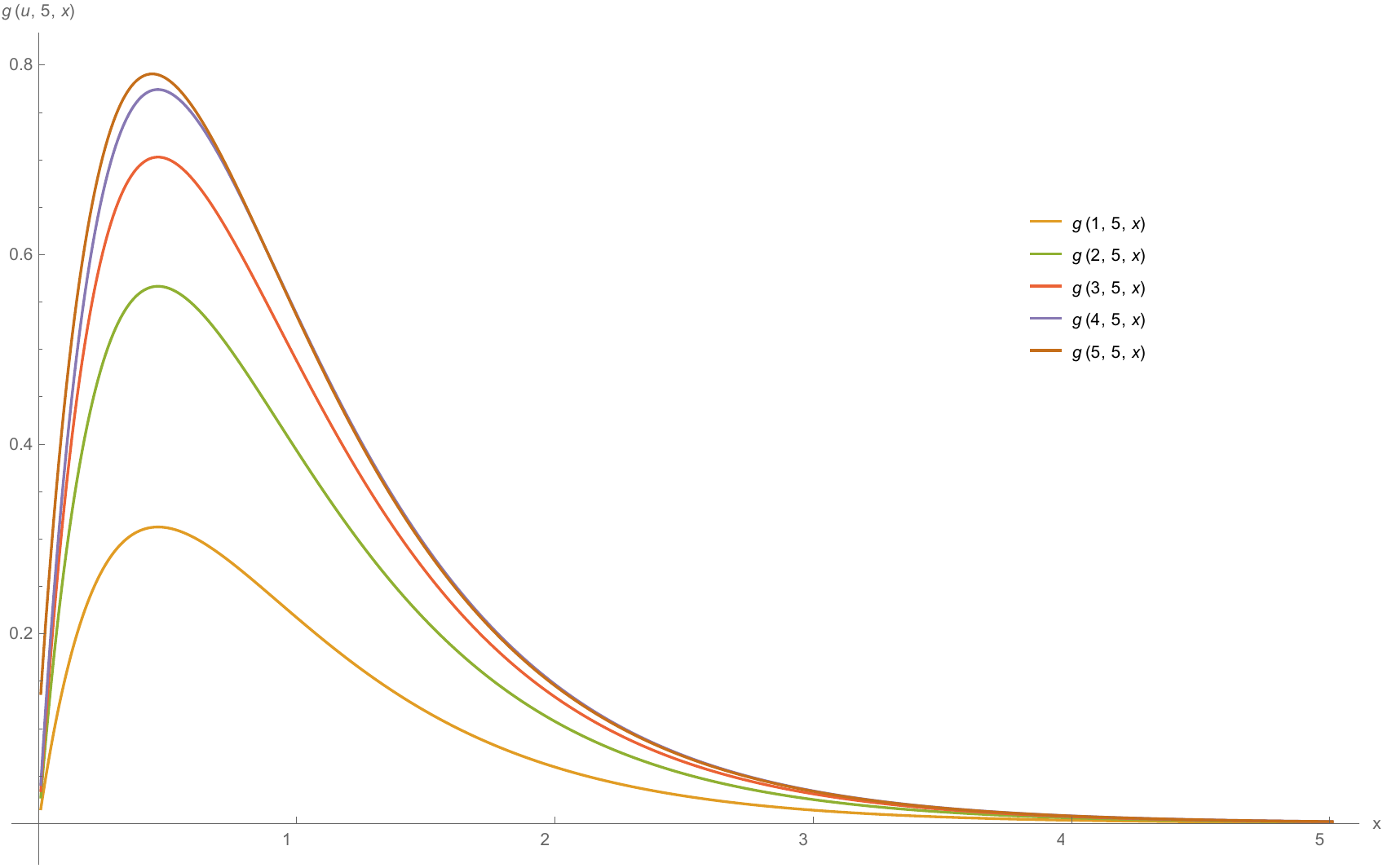}
	\includegraphics[width=.475\linewidth, height=150px]{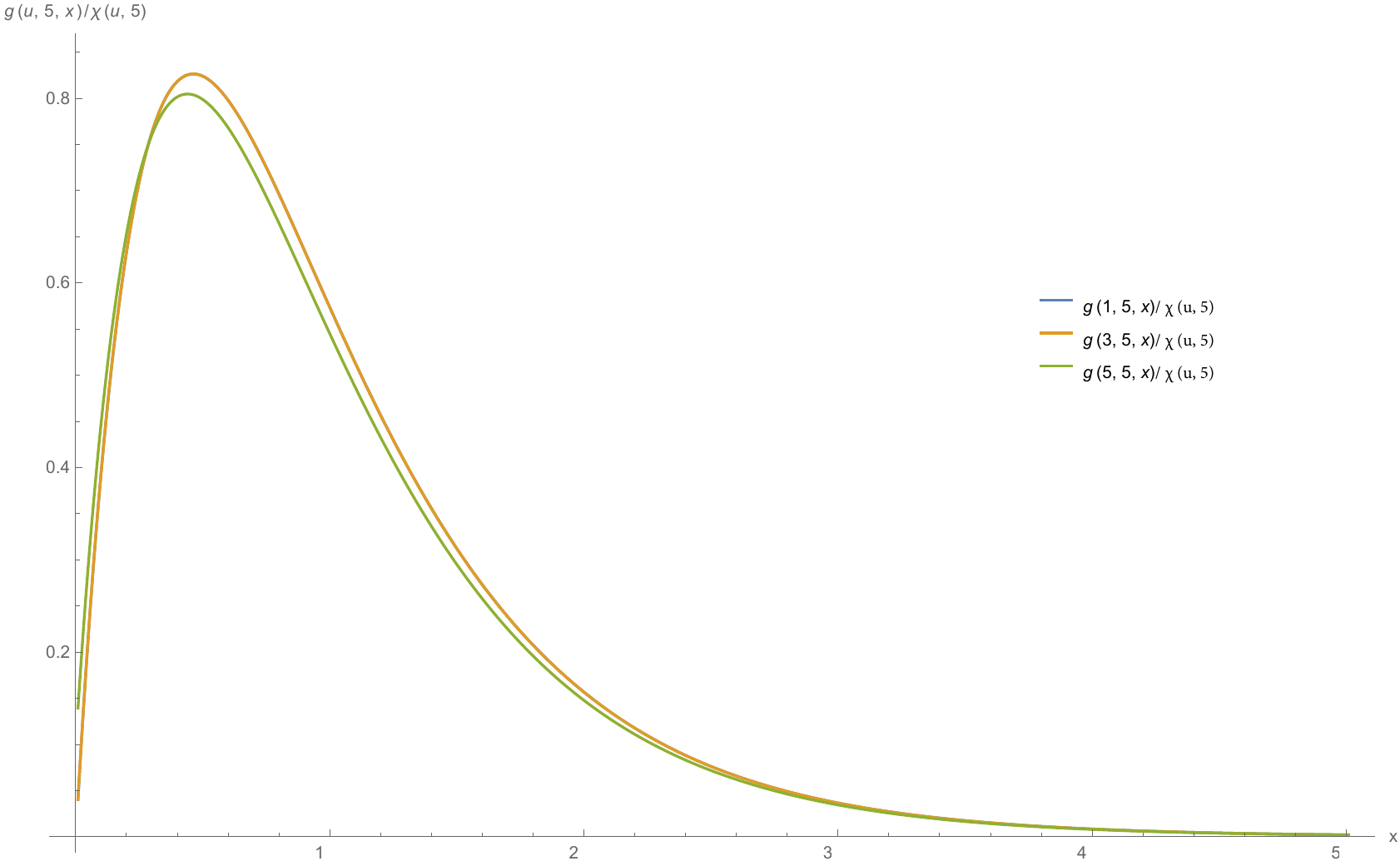}
	\caption{Densities $g(u,5,x)$ and $\tilde{g}(u,5,x)$, $u=1,2,\dots,5$, $c=0.75$}
	\label{F:g(u,5,x)-densi-comb}
\end{figure}	
\clearpage
\ren{
\begin{table}[h]
	\centering
	\begin{tabular}{lllllllll}
		\hline \hline
		\textbf{$m$\textbackslash{}$u$} & \textbf{0} & \textbf{0.2} & \textbf{0.5} & \textbf{0.7} & \textbf{1} & \textbf{3} & \textbf{5}  & \textbf{10}                    \\ \hline
		\textbf{0}                  & 1 & 0.938448   & 0.735759   & 0.591833  & 0.406006  & 0.017351 & 0.000499 & $4.33\times10^{-8}$ \\
		\textbf{1}                  & 0 & 0.029024  & 0.109012   & 0.151509  & 0.183614  & 0.045564 & 0.003096  & $9.16\times10^{-7}$ \\
		\textbf{2}                  & 0 & 0.008061 & 0.036744  & 0.058105 & 0.085016 & 0.057740 & 0.007781   & $6.84\times10^{-6}$ \\
		\textbf{3}                  & 0 & 0.004065 & 0.019192  & 0.031352 & 0.048728 & 0.056886 & 0.012810     & 0.0000288                   \\
		\textbf{4}                  & 0 & 0.002546 & 0.012179  & 0.020153 & 0.032137 & 0.051174 & 0.016812   & 0.0000840                   \\
		\textbf{5}                  & 0 & 0.001784  & 0.008590 & 0.014309 & 0.023133 & 0.044676 & 0.019397   & 0.0001904 
		\\ \hline \hline
	\end{tabular}
	\caption{Probability of having $m$ gains prior to ruin, $q(u, m)$, Combination of Exponentials $\lambda=2, c=1$ }
	\label{T:ex_number_gains_comb}
\end{table}
\begin{figure}[h]
	\centering
	\includegraphics[width=12cm]{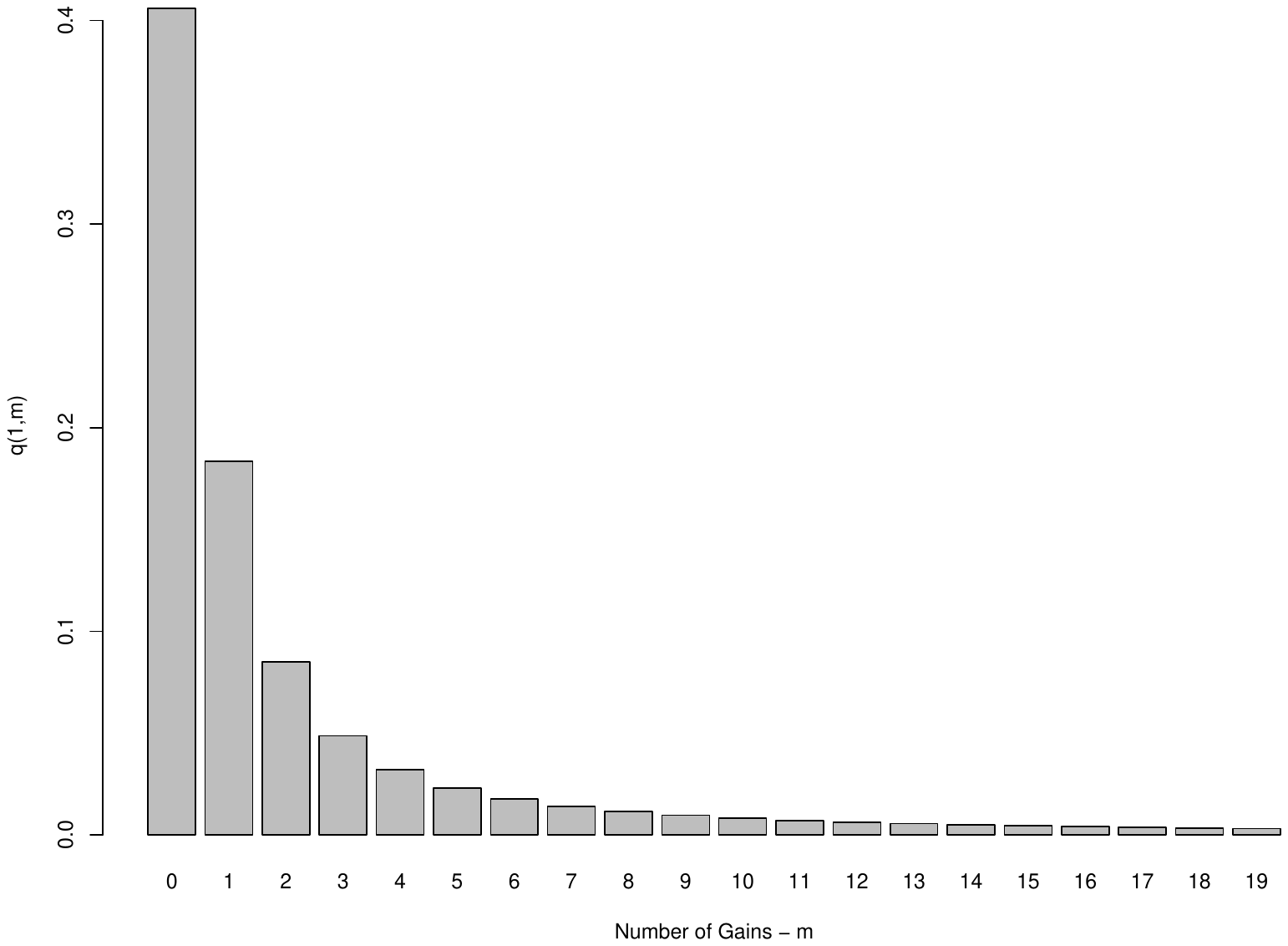}
	\caption{Probability of having $m$ gains prior to ruin, $q(1, m)$ - Combination of Exponentials}
	\label{F:q(1,m)-comb}
\end{figure}
}
\clearpage
%
\begin{landscape}
	\begin{table}
		\centering
		\begin{tabular}{llllllllll}
			\hline \hline
			\multicolumn{1}{c}{\textbf{m\textbackslash{}u}} & \multicolumn{1}{c}{\textbf{0}} & \multicolumn{1}{c}{\textbf{0.2}} & \multicolumn{1}{c}{\textbf{0.5}} & \multicolumn{1}{c}{\textbf{0.7}} & \multicolumn{1}{c}{\textbf{1}} & \multicolumn{1}{c}{\textbf{2}} & \multicolumn{1}{c}{\textbf{3}} & \multicolumn{1}{c}{\textbf{4}} & \multicolumn{1}{c}{\textbf{5}}\\ \hline
			\textbf{1}                  & 0.00045    & 0.00061      & 0.00096      & 0.00129      & 0.00203    & 0.00907    & 0.04024    & 0.17176    & 0.59775    \\
			\textbf{2}                  & 0.00319    & 0.00411      & 0.00598      & 0.00765      & 0.01101    & 0.03472    & 0.09187    & 0.15269    & 0.14352    \\
			\textbf{$\geq$3}   & 0.99636    & 0.99528      & 0.99307      & 0.99106      & 0.98696    & 0.95620    & 0.86789    & 0.67555    & 0.25873   \\ \hline \hline
		\end{tabular}
		\caption{Probability of having $m$ gains prior to reach an upper target $b=5$, $r(u,5,m)$, Combination of Exponentials, $\lambda=2$ and $c=0.75$}
		\label{T:ex_number_target_b5_comb}
	\end{table}
	\begin{table}
		\centering
		\begin{tabular}{llllllllll}
			\hline \hline
			\multicolumn{1}{c}{\textbf{m\textbackslash{}u}} & \multicolumn{1}{c}{\textbf{0}} & \multicolumn{1}{c}{\textbf{0.2}} & \multicolumn{1}{c}{\textbf{0.5}} & \multicolumn{1}{c}{\textbf{0.7}} & \multicolumn{1}{c}{\textbf{1}} & \multicolumn{1}{c}{\textbf{3}} & \multicolumn{1}{c}{\textbf{5}} & \multicolumn{1}{c}{\textbf{8}} & \multicolumn{1}{c}{\textbf{10}}\\ \hline
			\textbf{1}                  & 2.51$\times10^{-7}$                       & 3.38$\times10^{-7}$                         & 7.16$\times10^{-7}$                         & 1.12$\times10^{-6}$                         & 2.26$\times10^{-5}$                       & 0.00045                        & 0.04024                        & 0.17176                        & 0.59775                        \\
			\textbf{2}                  & 3.80$\times10^{-6}$                       & 5.01$\times10^{-6}$                         & 1.00$\times10^{-5}$                         & 1.52$\times10^{-5}$                         & 0.00023                        & 0.00319                        & 0.09187                        & 0.15269                        & 0.14352                        \\
			\textbf{$\geq$3}   & 0.999996                       & 0.99999                          & 0.99999                          & 0.99998                          & 0.99975                        & 0.99636                        & 0.86789                        & 0.67555                        & 0.25873                       
			         \\ \hline \hline
		\end{tabular}
		\caption{Probability of having $m$ gains prior to reach an upper target $b=10$, $r(u,10,m)$, Combination of Exponentials, $\lambda=2$ and $c=0.75$}
		\label{T:ex_number_target_b10_comb}
	\end{table}
	\begin{table}
		\centering
		\begin{tabular}{clllllllll}
			\hline \hline
			\textbf{m\textbackslash{}u} & \multicolumn{1}{c}{\textbf{0}} & \multicolumn{1}{c}{\textbf{0.2}} & \multicolumn{1}{c}{\textbf{0.5}} & \multicolumn{1}{c}{\textbf{0.7}} & \multicolumn{1}{c}{\textbf{1}} & \multicolumn{1}{c}{\textbf{2}} & \multicolumn{1}{c}{\textbf{3}} & \multicolumn{1}{c}{\textbf{4}} & \multicolumn{1}{c}{\textbf{5}} \\ \hline
			\textbf{1}                  & 0.00058                        & 0.00079                          & 0.00124                          & 0.00167                          & 0.00262                        & 0.01171                        & 0.05186                        & 0.21978                        & 0.59775                        \\
			\textbf{2}                  & 0.00479                        & 0.00617                          & 0.00899                          & 0.01151                          & 0.01660                        & 0.05290                        & 0.14352                        & 0.26768                        & 0.17958                        \\
			\textbf{$\geq$3}   & 0.99463                        & 0.99304                          & 0.98978                          & 0.98682                          & 0.98078                        & 0.93539                        & 0.80462                        & 0.51254                        & 0.22268                       
			      \\ \hline \hline                 
		\end{tabular}
		\caption{Probability of having $m$ gains prior to reach an upper target $b=5$, $r(u,5,m)$, Combination of Exponentials, $\lambda=2$ and $c=0.5$}
		\label{T:ex2_number_target_b5_comb}
	\end{table}
	\begin{table}
		\centering
		\begin{tabular}{clllllllll}
			\hline \hline
			\textbf{m\textbackslash{}u} & \multicolumn{1}{c}{\textbf{0}} & \multicolumn{1}{c}{\textbf{0.2}} & \multicolumn{1}{c}{\textbf{0.7}} & \multicolumn{1}{c}{\textbf{1}} & \multicolumn{1}{c}{\textbf{3}} & \multicolumn{1}{c}{\textbf{5}} & \multicolumn{1}{c}{\textbf{8}} & \multicolumn{1}{c}{\textbf{9}} & \multicolumn{1}{c}{\textbf{10}} \\ \hline
			\textbf{1}                  & 3.24$\times10^{-7}$                       & 4.37$\times10^{-7}$                         & 9.25$\times10^{-7}$                         & 1.45$\times10^{-6}$                         & 2.91$\times10^{-5}$                    & 0.00058                        & 0.05186                        & 0.21978                        & 0.59775                        \\
			\textbf{2}                  & 5.64$\times10^{-6}$                       & 7.46$\times10^{-6}$                         & 1.49$\times10^{-5}$                         & 2.26$\times10^{-5}$                         & 0.00035                        & 0.00479                        & 0.14352                        & 0.26768                        & 0.17958                        \\
			\textbf{$\geq$3}   & 0.99999                        & 0.99999                          & 0.99998                          & 0.99998                          & 0.99962                        & 0.99463                        & 0.80462                        & 0.51254                        & 0.22268                       
			     \\ \hline \hline                 
		\end{tabular}
		\caption{Probability of having $m$ gains prior to reach an upper target $b=10$, $r(u,10,m)$, Combination of Exponentials, $\lambda=2$ and $c=0.5$}
		\label{T:ex2_number_target_b10_comb}
	\end{table}
\end{landscape}

\end{document}